\documentclass[12pt]{amsart}
\usepackage{amssymb,amsmath}
\usepackage{tikz}
\usepackage{enumitem}
\usepackage{fullpage}

\newtheorem{theorem}{Theorem}[section]
\newtheorem{proposition}[theorem]{Proposition}

\newtheorem{lemma}[theorem]{Lemma}

\theoremstyle{definition}

\newtheorem{remark}[theorem]{Remark}

\numberwithin{equation}{section}

\makeatletter
\@namedef{subjclassname@2010}{%
  \textup{2010} Mathematics Subject Classification}
\makeatother

\title[Constructions for the  three-distance  theorem]{Some constructions for the higher-dimensional  three-distance  theorem}

\author[V. Berth\'e]{Val\'erie Berth\'e}
\address{IRIF, CNRS UMR 8243, Universit\'e Paris Diderot -- Paris 7, Case 7014, 75205 Paris Cedex 13, France}
\email{berthe@irif.fr}

\author[D. H. KIm]{Dong Han Kim}
\address{Department of Mathematics Education
Dongguk University - Seoul
30 Pildong-ro 1-gil, Jung-gu
Seoul, 04620 Korea}

\email{kim2010@dongguk.edu}
\date{}

\thanks{This work was supported by the Agence Nationale de la Recherche through the project   DynA3S (ANR-13-BS02-0003),
National Research Foundation of Korea (NRF-2015R1A2A2A01007090) and  by the Fondation Sciences Math\'ematiques de Paris (FSMP)  through the  support for  the  visit of  Dong Han Kim  at IRIF in 2016.}

\begin{document}

\baselineskip=17pt

\begin{center}
This paper is dedicated to Robert Tijdeman  on the occasion \\  of his 75th birthday.
\end{center}

\begin{abstract}
For a given real number  $\alpha$, let us place the fractional parts of  the points $0, \alpha, 2 \alpha,$
$ \cdots, (N-1) \alpha$
on the unit circle. These points partition the unit circle into  intervals  
having at most three lengths, one being the sum of the other two. This is the three distance theorem.
We consider  a  two-dimensional version  of the three distance theorem obtained by placing on the unit circle  the points
$ n\alpha+ m\beta $, for $0 \leq n,m < N$.
We  provide  examples of      pairs of real numbers $(\alpha,\beta)$, with $1,\alpha, \beta$    rationally independent,   for which  there  are  finitely many  lengths between  successive points  (and in fact, seven lengths), with  $(\alpha,\beta)$    not  badly approximable, as well
as examples for which there   are   infinitely many  lengths. \end{abstract}

\subjclass[2010]{Primary 11J13; Secondary 11J70, 11J71, 11B75, 11A55}

\keywords{Gap theorems, continued fractions, distribution modulo 1}

\maketitle
\section{Introduction}

For a given real number  $\alpha$ in $(0,1)$, let us place the points $\{0\}, \{\alpha \},$ 
$\{2 \alpha\},$$\cdots,\{ (N-1) \alpha\}$ on the unit circle, where $\{x\}$ denotes as usual the fractional part of $x$. These points partition the unit circle into $N$ intervals  
having at most three lengths, one being the sum of the other two.
This property is known as the {\em three distance theorem} and can be seen as 
a geometric interpretation
of  good 
approximation properties of  the Farey partial convergents in the 
continued fraction expansion of $\alpha$. In the literature, this theorem is  called
the Steinhaus theorem,      the three length, the three gap,
 or else,  the three step theorem.

The three distance theorem was 
initially
conjectured by Steinhaus, first proved  V. T. S\'os
 \cite{Sos:1958} and    Sur\'anyi \cite{Suranyi:1958},   and then by
  Slater \cite{Slater:1964}, \'Swierczkowski \cite{Swier:59},
 Halton \cite{Halton:1965}.  A survey of the different approaches used by these authors is to be found for instance 
in \cite{Albe,Ravenstein:1988,Slater:1967,Langevin:1991}. More recent proofs have also been given  in 
  \cite{Ravenstein:1988,Langevin:1991},  or in   \cite{Marklof17}   relying  on  the  properties of  space of two-dimensional Euclidean lattices. 
 See also \cite{Bleher:1991,PSZ:2016} for the  study of the  limiting distribution of the gaps.

 There  exist  numerous generalizations  of the three gap theorem.
Let us quote   for instance generalizations     for  groups \cite{FriedSos:1992},  for some isometries of compact Riemannian manifolds  \cite{BringerSchmidt:2008},
or else for interval exchange transformations \cite{Taha}.
Among generalizations, there are  two  natural   Diophantine frameworks that are dual, namely distance theorems  for toral translations
on the $d$-dimensional  torus ${\mathbb T}^d$   (see e.g. \cite{Chev:2007,Chevallier:2014,Vijay:2008}),  and  distance theorems  
 for   linear forms in $d$ variables on the one-dimensional torus  ${\mathbb T}$. This is  the framework of the present paper,
where we focus on  linear forms in two variables,  and   consider
 points 
$m\alpha+n \beta$, for $0 \leq n,m < N$,  in   ${\mathbb T}$.

This  generalization has been  considered  by Erd\"os  (as recalled in \cite{GeelenSimpson}) and  also  in 
 \cite{Liang:1979,ChungGraham,GeelenSimpson,FH:95,Chevallier,Bleher:12,HaynesMarklof17}.
 See also  \cite{CGVZ:2002} for  the   number of so-called primitive gaps.
In particular,  the following is proved  in  \cite{Chevallier}.  Let  $\alpha_1, \ldots,\alpha_d \in {\mathbb T}$ ($d \geq 3$) and
$2 \leq n_1 \leq \ldots \leq n_d $ be  integers.
The set
$\left\{\sum_{i=1}^{d} k_i \alpha_i,\ 0 \leq k_i < n_i, \ i=1,\ldots,d\right\}$
divides ${\mathbb T}$ into intervals  whose lengths   take at most $\prod_{i=1}^ {d-1} n_i +3 \prod_{i=1}^ {d-2}n_i +1$  values.
When $d=2$,   the upper  bound  is $N+3$  for  the case of interest here ($m\alpha+n \beta$, for $0 \leq n,m < N$), as proved in \cite{GeelenSimpson}.

There are natural cases where it is known that  the number of distances  is bounded (with respect to $N$, for the points 
$m\alpha+n \beta$, for $0 \leq n,m < N$).   This is the case if $ 1, \alpha, \beta$ are rationally dependent (this has been  proved by Holzman,
 as recalled in \cite{GeelenSimpson}).
Badly approximable  vectors  $(\alpha,\beta)$  have also  been proved by Boshernitzan and Dyson to   produce   a finite  number of distances.
 For a proof, see  \cite{Bleher:12}.
Nevertheless,   it is proved in   \cite{HaynesMarklof17} that 
 the number of  lengths  is generically unbounded, with an approach via homogeneous dynamics based on the  ergodic properties of the  diagonal action on the space of  lattices.
 However,  no explicit examples of this generic situation  were  known.  The  object of the present paper is  to  construct such examples.

Our main result is the following.
\begin{theorem}\label{thm}
Consider the set 
$ E_{N} (\alpha,\beta)  := \{ n\alpha + m\beta \in \mathbb T \,  :\,  0 \le n,m < N \}$, and let 
 $\Delta (E_{N} (\alpha,\beta))$   stand for  the set of distances between neighbor points of $E_{N} (\alpha,\beta)$.
We provide  effective constructions for  the following existence results.

\begin{enumerate}
\item[(i)] There exist   $(\alpha, \beta)$,  with $1, \alpha, \beta$  rationally independent   and   $(\alpha, \beta)$  not badly approximable,   such that:
$$    \forall  N , \  \# \Delta( E_{N} (\alpha,\beta)  ) \le 7. $$

\item[(ii)]  There exist  $(\alpha, \beta)$, with  $1, \alpha, \beta$   rationally independent,  such that:
$$\limsup_{N \to \infty} \# \Delta( E_{N} (\alpha,\beta)  ) = \infty. $$

\end{enumerate}
\end{theorem}


Our proof  avoids the use of a higher-dimensional analogue of continued fractions.  We rely on the  (regular) continued fraction expansions of $\alpha$ and $\beta$ and we combine several `rectangular'  levels of points of the form $n  \alpha + m \beta$,  for  $0 \leq n <  N$ and  $0 \leq m  <  M$, where $N$ or $M$ is  a
   denominator of   a principal  convergent  of  $\alpha$ or  $\beta$.

We ask  the  question   of the minimality   of the number of lengths:  is it possible to find $(\alpha,\beta)$, 
with $1,\alpha, \beta$ rationally independent, such that   $      \# \Delta( E_{N} (\alpha,\beta)  ) \le 6$, for all $N$?

 As an application and motivation for this theorem, one deduces  results on frequencies of square factors in two-dimensional Sturmian words,  such as studied in  \cite{BertheVuillon,BertheTijdeman:2003}.
Two-dimensional   Sturmian words are  defined as codings  of ${\mathbb Z}^2$-actions by rotations on the one-dimensional torus ${\mathbb T}$. More precisely,
 let $\alpha,\beta,\rho$ be real numbers, with
$1,\alpha,\beta$ rationally independent, and $0<\alpha +\beta <1$.
A  two-dimensional Sturmian word
over the three-letter alphabet $\{1,2,3 \}$ (with parameters
$\alpha,\beta,\rho$) is defined 
 as
a function $f: {\mathbb Z} ^2 \rightarrow
\{1,2,3\}$, with,  
for all  $(m,n) \in {\mathbb Z} ^2$, $(f(m,n)=i \Longleftrightarrow
m\alpha +n \beta+ \rho \in I_i \mbox{ modulo }1),$
where  
either $ I_1=[0,\alpha),\ I_2=[\alpha,\alpha +\beta), \ I_3=[\alpha +\beta,1),$ or 
$I_1=(0,\alpha],\ I_2=(\alpha,\alpha +\beta], \ I_3=(\alpha +\beta,1].$
According to  \cite{BertheVuillon}, frequencies of square factors of  size $N$   are equal to the lengths obtained by  putting  on ${\mathbb T}$ the points
 $ - n \alpha  -m \beta$, for  $-1 \leq n \leq N-1$,   $ 0 \leq m \leq N$.
 One thus has a correspondence between lengths and  frequencies, whereas 
 gap theorems correspond to  return words.
 Note that the convergence toward frequencies  (expressed in terms of balance properties) has been considered in \cite{BertheTijdeman:2003}.
 More generally, for   results of the same flavor for cut and project sets generalizing the Sturmian framework, see \cite{HKWS:2016,HJKW17}.

 \subsection*{Contents of the paper}
  Let us briefly sketch the contents of this paper.
 Notation   are introduced in Section \ref{sec:firstlevel} together with a  basic lemma  (Lemma \ref{lem:exchange}) that allows one to  express  in a convenient way the
 clockwise neighbor of
 a  point of the form  $n \alpha+ m \beta$.
 A  construction providing    pairs $(\alpha,\beta)$   with   a bounded number of lengths is described in Section \ref{sec:bounded}, while
 the case of  an unbounded number of lengths is handled in Section \ref{sec:unbounded}:
  Statement (i) of Theorem \ref{thm} is proved in Section \ref{sec:bounded}, and  Statement (ii)  in Section \ref{sec:unbounded}.
 
 \subsection*{Acknowledgements}
 Our deepest gratitude goes to Robert Tijdeman who has been  a constant source of inspiration
   for his input in word combinatorics and discrete mathematics  through his deep and wide 
 understanding of  equidistribution theory.  We  would like to thank Alan Haynes for  pointing out the problem  (see \cite{OP:2016})
and  for stimulating discussions.
 We also  would like to thank  Damien  Jamet and Thomas Fernique for  computer simulations.    Lastly, we  gratefully thank the referee  for  his  valuable comments which helped to improve the manuscript. 

\section{Preliminaries}\label{sec:firstlevel}


Let $\mathbb T = \mathbb R / \mathbb Z$.
Let $\alpha, \beta$ be real  numbers  in $(0,1)$. 
{\em We  assume  in all this paper that  $1,  \alpha, \beta$ are rationally independent. }

For $q,q'$ positive integers,  we define
$$ E_{q,q'} (\alpha, \beta) := \{ n\alpha + m\beta \in \mathbb T\,  :\,  0 \le n < q, \ 0 \le m < q' \},$$
and when $q=q'$,  we  use the notation 
$E_{N} (\alpha,\beta) : = E_{q,q'} (\alpha, \beta) $, with $N:=q=q'$.  We furthermore consider  $$ {\mathcal E}_{q,q'}(\alpha,\beta) := \{ (n, m) \,  : \, 0 \le n < q, \ 0 \le m < q' \}.$$
We will also use the shorthand notation $ E_{q,q'}$,  $E_{N}$ and $ {\mathcal E}_{q,q'}$.

Points  in $E_{q,q'} (\alpha, \beta) $ are considered  as  positioned on the unit circle   oriented clockwise  endowed with the origin point  $0$.
The point $n\alpha + m\beta$   is thus considered as positioned  at distance $\{n\alpha + m\beta\}$ from the  origin point $0$.   
The point that is located  immediately   after $n\alpha + m\beta$ clockwise  on the unit circle, that is,  its  clockwise  {\em neighbor},  is denoted as  
 $\Phi_{q,q'}(n\alpha+ m\beta)$,
or $\Phi (n\alpha+ m\beta)$, if there is no  confusion.
This thus defines a  map $\Phi_{q,q'}$ on $E_{q,q'}$  called the {\em neighbor map}.
For two points $a,b$ in  $\mathbb T$, the interval $(a,b)$ in $\mathbb T$ corresponds to the interval considered clockwise  on the unit circle with  endpoints being respectively  $a$ and  $b$. 
The set $E_{q,q'}  $ thus partitions the  unit circle  into  disjoint  intervals
 $( n\alpha + m\beta, \Phi_{q,q'} (n\alpha+ m\beta))$, for $(n,m) \in {\mathcal E}_{q,q'} $.

For a finite subset $E$ of  $\mathbb T = \mathbb R / \mathbb Z$,  we denote  by $\Delta (E)$  the set of distances between neighbor points of $E$ (again with 
distances being  counted  clockwise).
 For any  $(n,m) $ in  ${\mathcal E}_{q,q'}$, $\Delta_{q,q'}(n,m)$ (or   $\Delta(n,m)$ if there is no confusion) stands for  the distance   between $n\alpha+ m\beta$ and  $\Phi_{q,q'} (n \alpha+ m \beta)$.

For  any positive integer $q$, we  define the nonnegative integer  $|n|_q$ as 
$$ |n|_q \equiv n  \pmod q,   \quad  \mbox{ and }   \quad  0\le |n|_q < q.$$
We will   consider  in the following the map $ n \mapsto |n+r|_q$, for  a given  integer $r$.
In particular,   if  $0 \le r< q$,   one has  $|n+r|_q = n+r$ if $ 0 \leq n < q - r$, and  $ |n+r|_q = n+r-q$ if $ q-r  \leq n < q $.

Let $q,q'$ be two   given positive integers.  We also  assume that  $q$ and $q'$ are coprime. Then,
for any   integers  $r,r'$ such that  $\textrm{gcd}(r,q)=1$, $\textrm{gcd}(r',q')=1$,  $0< | r|  < q$, $0< | r'|  < q'$,  the map  
$$ \varphi  _{q,q'}\colon {\mathcal E}_{q,q'}  \rightarrow {\mathcal E}_{q,q'} , \qquad  (n, m) \ \mapsto   \   \left( | n+r|_q ,   |m + r'|_{q'} \right) $$ 
is  a   cyclic permutation of  ${\mathcal E}_{q,q'}.$ 
We thus will be able to describe  the elements  of the set  ${\mathcal E}_{q,q'}$
as   the  elements  of the orbit of $(0,0)$ under  the map $\varphi_{q,q'}$. In particular,  for each $(n,m)$, with $0 \le n < q$, $0 \le m < q'$,  there exists a unique $k$,  with $0 \le k < q q'$,
satisfying
$(n,m) = \left(|kr|_{q}, |kr'| _{q'} \right)$.  
Indeed, since  $1, \alpha, \beta$ are rationally independent, the following  map $\phi_{q,q'}$ acting  on $E_{q,q'} (\alpha, \beta)$, and defined by\footnote{This map  is well-defined since   $1, \alpha, \beta$ are rationally independent.}
\begin{equation} \label{eq:varphi}
\phi_{q,q'} (n \alpha + m \beta)= \langle \varphi_{q,q'} (n,m), (\alpha,\beta)\rangle    
\end{equation}
 is easily seen to be    injective, and thus surjective.

Let $(a_i)_{i \geq 1} $, $(a'_j)_{j \geq 1} $ stand for the respective sequences  of  partial quotients of $\alpha$ and $\beta$ in their continued fraction expansions, and 
denote by $(q_i)_{i \geq 1}$, $(q'_j)_{j \geq 1}$ the denominators of  their principal convergents.
Note that we will make a strong use of 
\begin{equation}\label{eq:frac}
q_k \| q_{k-1} \alpha \| + q_{k-1} \| q_k \alpha\| = 1.
\end{equation}
Here we denote $\| t\|$ by the distance to the nearest integer of $t \in \mathbb R$.

We now consider
$
E_{q_i,q'_j} (\alpha, \beta) = \{ n\alpha + m\beta \in \mathbb T : 0 \le n < q_i, 0 \le m < q'_j\},
$
for indexes $i,j$ for which we assume that they   satisfy 
$ q'_j = b' q_i +1$ for some positive integer $b'$.  Note that $ b' q'_{j-1}$ is coprime with $q'_j$ since $b'$ and  $q'_{j-1}$ are coprime with $q'_j$. We  take $r:= -(-1)^i q_{i-1}$ and $r':= (-1)^j b' q'_{j-1}$. 
We   consider   the following  cyclic permutations acting respectively on   ${\mathcal E}_{q_i,q'_j} (\alpha, \beta)$ and   $E_{q_i,q'_j} (\alpha, \beta)$:
$$ \varphi _{q_i,q'_j} \colon    n\alpha + m\beta \ \mapsto   (\left |  n - (-1)^i q_{i-1}\right |_{q_i}  ,\left | m + (-1)^j b' q'_{j-1} \right|_{q'_j} ) ,$$ 
$$ \phi _{q_i,q'_j} \colon    n\alpha + m\beta \ \mapsto   \  \left | n - (-1)^i q_{i-1}\right |_{q_i}  \,  \alpha + \left|m + (-1)^j b' q'_{j-1} \right|_{q'_j} \, \beta. $$
Lemma~\ref{lem:exchange}  below shows 
that, under Assumption  \eqref{eq:assumption}, the clockwise neighbor point $\Phi_{q_i,q'_j} (n \alpha+ m\beta)$ of $n \alpha+ m\beta$ in $ E_{q_i,q'_j} (\alpha, \beta)$ is exactly  $\phi (n \alpha+ m\beta)$, by using the   shorthand notation $\phi= \phi_{q_i,q'_j}$.
Lemma~\ref{lem:exchange} 
will be used in   the proofs of both statements of Theorem \ref{thm}. 
In particular, it will   play   a crucial role  in  Section~\ref{sec:bounded}
for the   case of a bounded number of lengths. 
Indeed, in order to count the  number of   lengths for a square set of points ${\mathcal E}_N$, we  consider several rectangular subsets of   points in  ${\mathcal E}_{N}$,  i.e., several  levels of points in $E_N$, with 
the points of $E_{q_i,q'_j}$ corresponding  to the first level. Further levels  of  points  will then be inserted  or removed.
Note that Lemma~\ref{lem:exchange} provides a  case where  there are only 4 possible  lengths.

\begin{lemma}\label{lem:exchange}
Let $\alpha, \beta $ be real  numbers in $(0,1)$ such that
$1,\alpha,\beta$ are rationally independent.  Let
 $(q_i)_{i \geq 1}$, $(q'_j)_{j \geq 1}$  stand for the denominators of  their principal convergents.
We assume that  for some $i, j \ge 1$ $$  q'_j = b' q_i +1$$ 
for some positive integer $b'$.
Let
$$  \phi (n \alpha+ m \beta) : = 
\begin{cases}  
 | n +q_{i-1} |_{q_i}  \,  \alpha + |m - b' q'_{j-1}|_{q'_j} \, \beta     ,  &\text{ if $i$, $j$ are odd}, \\
 | n +q_{i-1} |_{q_i}  \,  \alpha + |m + b' q'_{j-1}|_{q'_j} \, \beta    ,  &\text{ if $i$ is odd, $j$ is even}, \\
 | n - q_{i-1} |_{q_i}  \,  \alpha + |m - b' q'_{j-1}|_{q'_j} \, \beta   ,  &\text{ if $i$ is even, $j$ is odd}, \\
 | n - q_{i-1} |_{q_i}  \,  \alpha + |m + b' q'_{j-1}|_{q'_j} \, \beta   ,  &\text{ if $i$, $j$ are even}.
\end{cases}
$$
Then $\phi$ is a  permutation of  $E_{q_i,q'_j} (\alpha, \beta)$.

Under the further  assumption that
\begin{equation} \label{eq:assumption}
\| q'_j \beta\| <  \| q_{i-1}\alpha \| - b'\| q'_{j-1}\beta \| , \end{equation} 
then the maps $\Phi$ and $\phi$ coincide, that is, the point  $\Phi( n\alpha + m\beta)$  that is located immediately after $ n\alpha + m\beta$ clockwise on the unit circle, for $0 \leq n < q_i$, $0 \leq m < q'_j$,
  is $\phi(n\alpha + m \beta)$.
Moreover, the distance (counted clockwise) $\Delta(n,m)$ between $n\alpha+ m\beta$
 and  $\phi (n \alpha+ m \beta)$, for $0 \leq n < q_i$, $0 \leq m < q'_j$,  takes one of the following values
\begin{align*}
&\| q_{i-1}\alpha \| - b'\| q'_{j-1}\beta \| ,&  &\| q_{i-1}\alpha \| - b'\| q'_{j-1}\beta \|  - \| q'_j \beta\|,  \\
&\| q_{i-1}\alpha \| - b'\| q'_{j-1}\beta \| + \| q_i \alpha\|, & &\| q_{i-1}\alpha \| - b'\| q'_{j-1}\beta \| + \| q_i \alpha\| - \| q'_j \beta\|.
\end{align*}

More precisely,  if  e.g. $i,j$  are odd, then    $\Delta(n,m)$ equals:
\begin{align*}
&\| q_{i-1}\alpha \| - b'\| q'_{j-1}\beta \|,  
 &&0 \le n < q_i - q_{i-1}, \, b' q'_{j-1} \le m < q'_j, \\
&\| q_{i-1}\alpha \| - b'\| q'_{j-1}\beta \| - \| q'_j \beta\|,  
 &&0 \le n < q_i - q_{i-1}, \, 0 \le m < b' q'_{j-1}, \\
&\| q_{i-1}\alpha \| - b'\| q'_{j-1}\beta \| + \| q_i\alpha\|,  
 &&q_i - q_{i-1} \le n < q_i, \,  b' q'_{j-1} \le m < q'_j, \\
&\| q_{i-1}\alpha \| - b'\| q'_{j-1}\beta \| + \| q_i \alpha\| - \| q'_j \beta\|, 
 &&q_i - q_{i-1} \le n < q_i, \, 0 \le m < b' q'_{j-1},
\end{align*}
and  if  e.g. $i,j$  are even, then    $\Delta(n,m)$ equals:
\begin{align*} 
&\| q_{i-1}\alpha \| - b'\| q'_{j-1}\beta \|,  
 && q_{i-1} \le n < q_i, \, 0 \le m < q'_j - b' q'_{j-1}, \\
&\| q_{i-1}\alpha \| - b'\| q'_{j-1}\beta \| - \| q'_j \beta\|,  
 && q_{i-1} \le n < q_i, \, q'_j - b' q'_{j-1} \le m < q'_j, \\
&\| q_{i-1}\alpha \| - b'\| q'_{j-1}\beta \| + \| q_i\alpha\|,  
 && 0 \le n < q_{i-1}, \, 0 \le m < q'_j - b' q'_{j-1}, \\
&\| q_{i-1}\alpha \| - b'\| q'_{j-1}\beta \| + \| q_i \alpha\| - \| q'_j \beta\|, \!\!
 && 0 \le n < q_{i-1}, \, q'_j - b' q'_{j-1} \le m < q'_j.
\end{align*} 
Similar formulas hold for the other cases.\end{lemma}

\begin{proof}
Recall  that  $\textrm{gcd}  (b'q'_{j-1}, q'_j) =1$.

We  first assume that $i$, $j$ are odd. 
Then,  $q_{i-1}\alpha - p_{i-1} = \|q_{i-1} \alpha \|$, $q'_{j-1}\beta - p'_{j-1} = \| q'_{j-1} \beta \|$ and  
$ q_{i} \alpha - p_i = -\| q_{i} \alpha \| $, $q'_{j} \beta - p'_j = - \| q'_{j} \beta \| $. Therefore,
$
q_{i-1} \alpha -b'  q'_{j-1} \beta 
=  \| q_{i-1}\alpha \| - b'\| q'_{j-1}\beta \|  +\left( p_{i-1} - b' p'_{j-1} \right).$
It follows that \begin{align*}
&\phi(n \alpha+ m \beta) - (n \alpha+ m \beta)  \\
&= |n+q_{i-1} |_{q_i} \,  \alpha +  |m-b'q'_{j-1}|_{q'_j}  \, \beta  - (n \alpha+ m \beta) \\
&= \begin{cases}
q_{i-1} \alpha -b'  q'_{j-1} \beta,  
 & 0 \le n < q_i - q_{i-1}, \ b'q'_{j-1} \le m < q'_j, \\
q_{i-1} \alpha - \left( b'  q'_{j-1} - q'_j\right) \beta,  
 & 0 \le n < q_i - q_{i-1}, \ 0 \le m < b'q'_{j-1}, \\
\left( q_{i-1} -q_i\right) \alpha -b'  q'_{j-1} \beta,  
 &  q_i - q_{i-1} \le n < q_i, \ b'q'_{j-1} \le m < q'_j, \\
\left( q_{i-1} -q_i\right) \alpha - \left( b'  q'_{j-1} - q'_j\right) \beta,  
 & q_i - q_{i-1} \le n < q_i, \ 0 \le m < b'q'_{j-1}.
\end{cases}\end{align*}
Let us assume that Equation~\eqref{eq:assumption} holds.
Let  $\widetilde{\Delta}(n,m)$ stand for  the distance (counted clockwise)  between $n\alpha+ m\beta$
 and  $\phi (n \alpha+ m \beta)$, for $0 \leq n < q_i$, $0 \leq m < q'_j$. 
Denote $D : = \| q_{i-1}\alpha \| - b'\| q'_{j-1}\beta \|$.  One has 
\begin{equation*}
\widetilde{\Delta}(n,m) = 
\begin{cases}
D,
 &0 \le n < q_i - q_{i-1}, \ b'q'_{j-1} \le m < q'_j, \\
D 
- \| q'_j \beta\|,  
 &0 \le n < q_i - q_{i-1}, \ 0 \le m < b'q'_{j-1}, \\
D 
+ \| q_i\alpha\|,   
 & q_i - q_{i-1} \le n < q_i, \  b'q'_{j-1} \le m < q'_j, \\
D 
+ \| q_i \alpha\| - \| q'_j \beta\|, 
 &q_i - q_{i-1} \le n < q_i, \ 0 \le m < b'q'_{j-1}.
\end{cases}\end{equation*}
By  Equation~\eqref{eq:assumption}, for all four cases, the values of  the right hand side are  positive and less than 1.
\medskip

We now assume that  $i$ is odd and $j$ is even. Then, 
$q_{i-1} \alpha + b' q'_{j-1} \beta = \Delta +\left( bp_{i-1} + b' p'_{j-1} \right).$ 
Similarly, we deduce that 
\begin{equation*}
\widetilde{\Delta}(n,m) 
= 
\begin{cases}
D , 
 &0 \le n < q_i - q_{i-1}, \, 0 \le m < q'_j - b'  q'_{j-1}  , \\
D 
- \| q'_j \beta\|,  
 &0 \le n < q_i - q_{i-1}, \,  q'_j - b'  q'_{j-1}  \le m < q'_j, \\
D 
+ \| q_i\alpha\| ,   
 &q_i - q_{i-1} \le n < q_i,\, 0 \le m < q'_j - b'  q'_{j-1} , \\
D 
+ \| q_i \alpha\| - \| q'_j \beta\|, \!
 &q_i - q_{i-1} \le n < q_i, \, q'_j - b'  q'_{j-1}  \le m < q'_j.
\end{cases}\end{equation*}

If $i$ is even and $j$ is odd, then 
\begin{equation*}
\widetilde{\Delta}(n,m) 
= 
\begin{cases}
D , 
 &q_{i-1} \le n < q_i,\ b'  q'_{j-1} \le m < q'_j, \\
D 
- \| q'_j \beta\|,  
 &q_{i-1} \le n < q_i,\ 0 \le m < b'  q'_{j-1}, \\
D 
+ \| q_i\alpha\| ,  
 &0 \le n < q_{i-1}, \ b'  q'_{j-1} \le m < q'_j, \\
D 
+ \| q_i \alpha\| - \| q'_j \beta\|, 
 &0 \le n < q_{i-1}, \ 0 \le m < b'  q'_{j-1}.
\end{cases}\end{equation*}
Lastly, if $i,j$ are even, then 
\begin{equation*}
\widetilde{\Delta}(n,m) 
= 
\begin{cases}
D , 
 &q_{i-1} \le n < q_i,\  0 \le m < q'_j -b'  q'_{j-1} , \\
D 
- \| q'_j \beta\|,  
 &q_{i-1} \le n < q_i ,\  q'_j -b'  q'_{j-1} \le m < q'_j, \\
D 
+ \| q_i\alpha\|,   
 &0 \le n < q_{i-1}, \ 0 \le m < q'_j -b'  q'_{j-1}, \\
D 
+ \| q_i \alpha\| - \| q'_j \beta\|, 
 &0 \le n < q_{i-1}, \ q'_j -b'  q'_{j-1} \le m < q'_j.
\end{cases}\end{equation*}
Hence, we conclude that, for the four cases  obtained  by considering  the parity of $i,j$,  one gets
\begin{align*}
&\# \left \{ (n,m) : \widetilde{ \Delta}(n,m) = D 
 \right\} = (q_i -q_{i-1})(q'_j - b'  q'_{j-1}) , \\
&\# \left \{ (n,m) :  \widetilde{\Delta}(n,m) = D 
-  \| q'_j \beta\| \right \} = (q_i -q_{i-1}) b'  q'_{j-1} , \\
&\# \left \{ (n,m) :  \widetilde{\Delta}(n,m) = D 
 +  \| q_i\alpha\| \right \} =  q_{i-1} ( q'_j -b'  q'_{j-1}), \\
&\# \left \{ (n,m) :  \widetilde{\Delta}(n,m) = D 
 + \| q_i \alpha\| - \| q'_j \beta\| \right \} =  q_{i-1} b'  q'_{j-1} .
\end{align*}

We now show the map $\phi$ sends a point to its neighbor point in the clockwise direction, that is, $\phi$ and $\Phi$ coincide.
It is sufficient to notice that the $q_i q'_j$ intervals $\big((n \alpha + m \beta), \phi (n \alpha+ m \beta) \big)$ of ${\mathbb T}$  never overlap. 
Indeed, the sum of their lengths $\widetilde{\Delta}(n,m)$ equals $1$, as  shown below
using \eqref{eq:frac}:
\begin{align*}
1 &=q'_j - b' q_i \\
&= \left(q_i \| q_{i-1} \alpha\| + q_{i-1} \| q_i \alpha \| ) q'_j - b' (q'_j \|q'_{j-1} \beta\| + q'_{j-1} \|q'_j \beta\| \right) q_i \\
&= q_i  q'_j \left(\| q_{i-1}\alpha \| - b'\| q'_{j-1}\beta \| \right) + q_{i-1} q'_j \|q_i \alpha\| - b' q'_{j-1} q_i  \|q'_j \beta \|\\
&= \sum_{(m,n)\in {\mathcal E}_{q_i,q'_j}} \left(  \phi (n \alpha+ m \beta) - (n \alpha + m \beta )  \right ).
\end{align*}
\end{proof}

\begin{remark}

The map  $ \varphi_{q_i,q'_j}$   (associated  with $\phi$ through
 \eqref{eq:varphi})   is   an exchange of  4 rectangles on $  {\mathcal E}_{q_i,q'_j} $.  
For  an illustration, see Figure \ref{fig:squarepointsbis} below.\end{remark}

\begin{figure}
\begin{center}
\begin{tikzpicture}[every loop/.style={}]
  \draw [thick, <->] (2,0) -- (0,0) -- (0,8);
  \draw (.3,0)--(.3,7);
  \draw (0,4.5)--(1.4,4.5);
  \draw (1.4,0)--(1.4,7)--(0,7);

  \node [below=4pt] at (0.4,0) {$q_{i-1}$};   
  \node [below=4pt] at (1.4,0) {$q_i$};  
  \node [below=4pt] at (2,0) {$n$};  
  \node [left=4pt] at (0,7) {$q'_j$};  
  \node [left=4pt] at (0,8) {$m$};  
  
   \draw[<->] (1.6,4.5) to (1.6,7);
  \node [right=4pt] at (1.6,5.75) {$b' q'_{j-1}$}; 
  
   \node at (3.3,4) {$\Longrightarrow$};
  \node at (3.3,4.5) {$\varphi_{q_i,q'_j}$};
  \draw [thick, <->] (6.6,0) -- (4.6,0) -- (4.6,8);
  \draw (5.7,0)--(5.7,7);
  \draw (4.6,2.5)--(6,2.5);
  \draw (6,0)--(6,7)--(4.6,7);

  \node [below=4pt] at (6,0) {$q_i$};  
 \node [below=4pt] at (6.6,0) {$n$};  
 \node [left=4pt] at (4.8,2.5) {$b'q'_{j-1}$};
  \node [left=4pt] at (4.6,7) {$q'_j$};   
  \node [left=4pt] at (4.6,8) {$m$};  
  
  \draw[<->] (5.7,7.2) to (6,7.2);
  \node [above=4pt] at (5.85,7.2) {$q_{i-1}$}; 
\end{tikzpicture}
\end{center}
\caption{The action  of $\varphi_{q_i,q'_j}$ on ${\mathcal E}_{q_i,q'_j} (\alpha,\beta)$ is  an exchange  of 4 sub-rectangles (here,  $i$ and $j$  are assumed to be  even).}
\label{fig:squarepointsbis}

\end{figure}
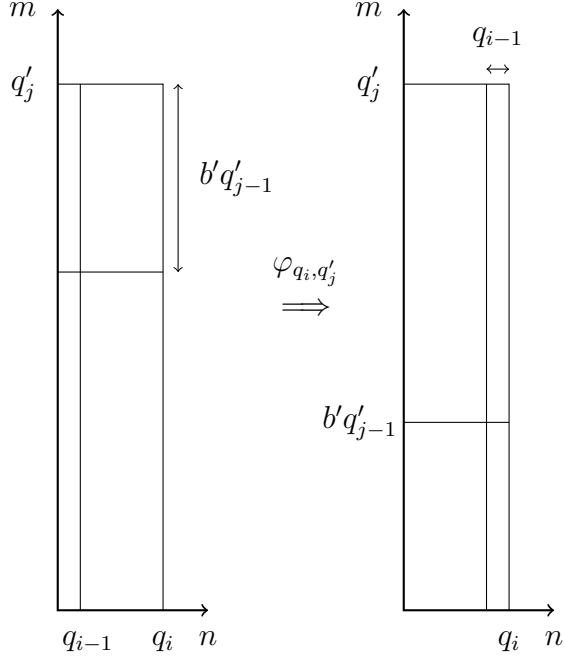

\begin{remark}\label{rem:primitive}
According to \cite{CGVZ:2002}, a distance in $\Delta(E_{q,q'} ( \alpha, \beta))$  is said to be  {\em primitive} if  it  is not a sum of shorter lengths (not necessarily distinct).
It is proved in  the same paper  \cite{CGVZ:2002} that there are at  most 4 primitive lengths for $E_{q,q'} ( \alpha, \beta)$. The  lengths given in  Lemma \ref{lem:exchange} are  primitive ones 
(with the assumption   that $1, \alpha$, $\beta$ are rationally independent).
\end{remark}

\section{Bounded  number of lengths}\label{sec:bounded}
This section is devoted to the proof of Statement (i) of Theorem \ref{thm}. 
We    provide a strategy for constructing examples of  pairs $(\alpha,\beta)$ providing a low number of distances 
$\Delta( E_{N} (\alpha, \beta))$, for all $N$.
We will rely on Lemma~\ref{lem:exchange}, and  use 
the existence of   positive  integers $b$ such that $q_i= bq'_j+1$,  as well as   the existence of  positive  integers $b'$  such that   
 $q'_i= b 'q_j+1$ for suitable $i,j$,   with 
$\alpha$ and $\beta$   playing  a symmetrical role.

\subsection*{Construction of the sequences of  convergents $(q_k)_k$ and $(q'_k)_k$.} 

We provide a construction of  sequences  of convergents $(q_k)_k$, $(q'_k)_k$, and  sequences 
$(b_k)_k,$  $(b'_k)_k$ such that  the following holds, for all $k \geq 1$:
\begin{equation}\label{31}
q'_k = b'_k q_k + 1, \qquad  q_{k+1} = b_{k+1} q'_k +1.
\end{equation}

Recall that  $q_{-1}=q'_{-1}=0$ and $q_0=q'_0=1$. We   then start with $q_1 = 3$, $q'_1 = (q_1)^3 + 1 =  28$ with $b'_1 =9$. Also $a_1=q_1=3$, $a'_1=q'_1=28$.
Let
$$ a_2 = ((q_1)^6 + q_0 - 1) b'_1 + (q_1)^5= 3^8 +  3^{5}, \qquad q_2 = 3^9 +3^6+ 1= 3^6 q'_1 + 1.$$
We  set  $b_2 = 3^6 = (q_1)^6 + q_0 - 1$. 
 
Assume now  that for some  index $k$,   one has $q'_k = b'_k q_k +1$. 
Choose $a_{k+1} = \left( (q_k)^6 + q_{k-1} -1 \right) b'_k + (q_k)^5$.
Then, we get
\begin{align*}
q_{k+1} &= a_{k+1} q_k + q_{k-1} 
= \left( (q_k)^6 + q_{k-1} -1 \right) b'_k q_k + (q_k)^6 + q_{k-1} \\
&= \left( (q_k)^6 + q_{k-1} -1 \right) \left(b'_k q_k + 1 \right) +1 
= \left( (q_k)^6 + q_{k-1} -1 \right) q'_k +1 .
\end{align*}
Let  $b_{k+1} = (q_k)^6 + q_{k-1} -1$. Then $q_{k+1} = b_{k+1} q'_k +1$.
Next, we set $a'_{k+1} = \left( (q'_k)^6 + q'_{k-1} -1 \right) b_{k+1} + (q'_k)^5$.
Then, we have similarly
\begin{align*}
q'_{k+1} &= a'_{k+1} q'_k + q'_{k-1} =\left( (q'_k)^6 + q'_{k-1} -1 \right) b_{k+1} q'_k + (q'_k)^6  + q'_{k-1} \\
&= \left( (q'_k)^6 + q'_{k-1} -1 \right) \left(b_{k+1} q'_k + 1 \right) +1 
= \left( (q'_k)^6 + q'_{k-1} -1 \right) q_{k+1} +1 
\end{align*}
and $b'_{k+1} = (q'_k)^6 + q'_{k-1} -1$.

In summary,  we inductively construct sequences $(q_k)_k$, $(q'_k)_k$ satisfying for any $k \ge 1$
\begin{equation}\label{eq:ab} \begin{split}
a_{k+1} &= \left( (q_k)^6 + q_{k-1} -1 \right) b'_k + (q_k)^5, \\
a'_{k+1} &= \left( (q'_k)^6 + q'_{k-1} -1 \right) b_{k+1} + (q'_k)^5,
\end{split}
\end{equation}
\begin{equation}\label{eq:ab2}
b_{k+1} = (q_k)^6 + q_{k-1} -1,  \qquad  b'_{k+1} = (q'_k)^6 + q'_{k-1} -1. 
\end{equation}
Then, for any $k \ge 1$, \eqref{31}  holds.
Note that 
\begin{equation}\label{32}
q_{k+1} = b_{k+1} q'_k +1
= b_{k+1} (b'_k q_k +1) +1 = b'_k ( b_{k+1} q_k) + ( b_{k+1} + 1).
\end{equation}

Since 
\begin{align*}
a_{k+1} &= \frac{q_{k+1}-q_{k-1}}{q_k} = \frac{b_{k+1} q'_k +1-q_{k-1}}{q_k} \\
&= \frac{ ( (q_k)^6 + q_{k-1} -1 ) q'_k +1-q_{k-1}}{q_k} \ge (q_k)^5 q'_k,
\end{align*}
we have
$$
q_{k-1} \| q_k \alpha \|< \frac{1}{a_{k+1}a_k} <  \frac{1}{2q'_k} < \| q'_{k-1} \beta \|
< \| q'_{k-1} \beta \| + q'_{k-1} \| q'_k \beta \|.
$$
Therefore, it follows that  
\begin{equation}\label{33}\begin{split}
\| q_{k-1} \alpha \| - b'_k \| q'_{k-1} \beta \| 
&>  \frac{1 - q_{k-1} \| q_k \alpha \|}{q_k} - \frac{b'_k}{q'_k} \\
&= \frac{q'_{k} - b'_k q_k}{q_k q'_k} - \frac{q_{k-1}}{q_k} \| q_k \alpha \|  
= \frac{1}{q_k q'_{k}}   - \frac{q_{k-1}}{q_{k}} \| q_k \alpha \| \\
&= \frac{q_{k+1} \| q_k \alpha \| + q_k \| q_{k+1} \alpha \|}{q_k q'_{k}} - \frac{q_{k-1}}{q_k} \| q_k \alpha \|  \\
&= \left( \frac{b_{k+1}}{q_k} + \frac{1}{q_k q'_{k}}  - \frac{q_{k-1}}{q_k} \right) \| q_k \alpha \| + \frac{\| q_{k+1} \alpha \|}{q'_{k}} \\
&= \left( (q_k)^5 - \frac{1}{q_k}  + \frac{1}{q_k q'_{k}}  \right)  \| q_k \alpha \|+ \frac{\| q_{k+1} \alpha \|}{q'_{k}} >0.
\end{split}\end{equation}

We also claim that 
\begin{equation}\label{eqq}
q_k < b'_k < (q_k)^3, \qquad   q'_k < b_{k+1} < (q'_k)^3.
\end{equation}
Indeed, if  $q_k < b'_k < (q_k)^3$, then,  using \eqref{eq:ab2},  we have
$$q'_k = b'_k q_k +1 < q_k^4 +1 < b_{k+1} < (q_k^2 +1)^3 < (b'_k q_k +1)^3 =  (q'_k)^3.$$
The choice of   $q_1,q'_1,b_1$  with $q_1 < b'_1 < q_1^3$    concludes the proof of the  claim.

\subsection*{Rational independence of $1,\alpha,\beta$} 

Suppose that $1,\alpha,\beta$ are rationally dependent.
Then, there exist integers $n_0, n_1, n_2$ satisfying
$n_0 + n_1 \alpha + n_2 \beta = 0$. Since $\alpha, \beta$ are both
irrational numbers, one has
$n_1, n_2 \ne 0$.
Then, for large $k$ such that 
$$|n_1 | <  \frac{q_{k+1}}{q'_k} \ \text{ and } \ 
|n_2| <  \frac{b'_{k+1}}{2}  <  b'_{k+1} q_{k+1} \| q_k \alpha\| < q'_{k+1} \| q_k \alpha \|,   $$ we have
$$
\| n_1 q'_{k} \alpha \| = \| n_2 q'_{k} \beta \| \le |n_2| \|q'_{k} \beta \| < \frac {|n_2|}{q'_{k+1}} < \| q_k \alpha\|.$$
This is a contradiction to the fact that $\| n \alpha \| > \| q_k \alpha\|$ for any $1 \le n < q_{k+1}$ (see for instance \cite[Chapter 1, Theorem 6]{Lang}).

Let us check now that $(\alpha, \beta)$  is not badly approximable.
 Recall that an irrational vector $(\alpha,\beta)$ is said to be badly approximable if there exists $C >0$ such that 
$$ \| n\alpha + m\beta \| > \frac{C}{|(n,m)|^2}$$
for any non-zero pair of integers $(n,m)$.
For the example constructed in this section, if $(n,m) = (q_k,0)$, then,  by \eqref{32}, one gets
$$
 \| q_k \alpha + 0 \beta \| = \| q_k \alpha \| < \frac{1}{q_{k+1}} < \frac{1}{b'_k b_{k+1} q_k} < \frac{1}{(q_k)^7} = \frac{1}{(q_k)^5} \frac{1}{|(q_k,0)|^2} .
$$
Therefore, $(\alpha,\beta)$ is not badly approximable.

\subsection*{Organization of the proof} 
We  first assume   $q_k < N \le q'_k$ and $k$ is even, and provide all the details  for this case.
The case $k$ odd and   then,  the  case $q_k < N \le q'_k$, will be   briefly discussed   at the end of the proof.

We  thus  assume   $q_k < N \le q'_k$ and $k$ is even (see Figure \ref{fig:org}). Note that $q'_{k} = b'_{k}q_k +1 > q_k$. The  proof will be divided into three steps.
\begin{itemize}
\item 
We first  describe    the lengths in  $E_{q_{k},q'_k}(\alpha,\beta)$. 
There are 4 lengths according to Lemma  \ref{lem:exchange}.
\item
Then, we will     deduce the description of the lengths in   $E_{q_k,N}(\alpha,\beta)$
 from the  description  of the lengths in  $E_{q_{k},q'_k}(\alpha,\beta)$. We reduce the set  of  points $(n,m)$ under consideration in this step.  Dynamically, this will
 correspond  to   induce the map  $\phi_{q_k,q'_k}$ (or  similarly the map  $\varphi_{q_k,q'_k}$).  We will go from 4 lengths to 6 lengths.
 \item Lastly,  the description of the lengths in   $E_{N}(\alpha,\beta)$ will be deduced 
 from the  description  of the lengths in  $E_{q_{k},N}(\alpha,\beta)$ by performing  an `exduction' step with  points $(n,m)$   being  inserted, creating a seventh length. 

 \end{itemize}
 \begin{figure} 
\begin{center}
\begin{tikzpicture}[every loop/.style={}]
  \draw [thick, <->] (5.3,0) -- (0,0) -- (0,8);
  \draw (4.5,0)--(4.5,4.5); 
  \draw (0,4.5)--(4.5,4.5);
  \draw (1.4,0)--(1.4,7)--(0,7);
    \node [below=4pt] at (1.4,0) {$q_k$};  
      \node [below=4pt] at (5.3,0) {$n$};  
        \node [left=4pt] at (0,7) {$q'_k$};  
  \node [left=4pt] at (0,8) {$m$};  
   \node [left=4pt] at (0,4.5) {$N$};
     \node [below=4pt] at (4.5,0) {$N$};
\end{tikzpicture}
\end{center}
 \caption{The sets ${\mathcal E}_{q_{k},q'_k}$,  ${\mathcal E}_{q_{k},N}$ and ${\mathcal E}_{N}$.}
\label{fig:org}
\end{figure}
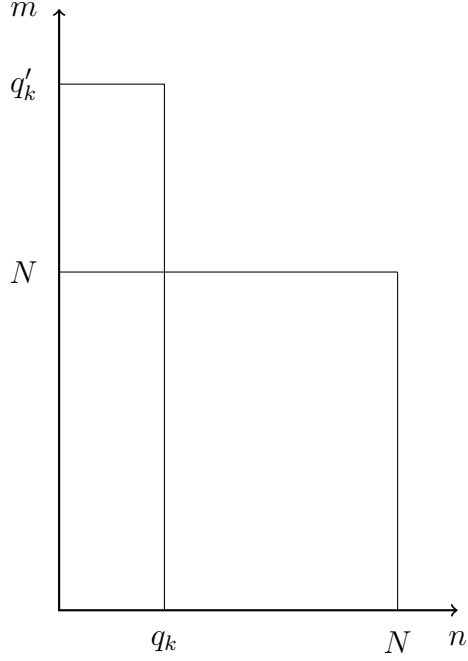

\subsection*{Distribution of the points of  $E_{q_k,q'_k}(\alpha,\beta)$} 

We apply Lemma~\ref{lem:exchange} for $E_{q_k,q'_k}(\alpha,\beta)$
by using the  fact   that $ q'_k=   b'_k q_k +1$.   With the notation of the lemma,
$i=j=k$, $b=b'_k$.  We are in the case $i,j$ even, since $k$ is even. 
Observe  that   Assumption \eqref{eq:assumption} holds, namely
$\| q'_k \beta\| <  \| q_{k-1}\alpha \| - b_k'\| q'_{k-1}\beta \| $. 
This comes from \eqref{33}, applied twice to get the  two left inequalities below:
\begin{equation*}
\| q_{k-1} \alpha \| - b'_k \| q'_{k-1} \beta \| 
>  \| q_k \alpha \| > b'_{k+1} \| q'_k \beta \| \ge  \| q'_k \beta \|.
\end{equation*}

By Lemma  \ref{lem:exchange}, the neighbor map  $ \Phi_{q_k,q'_k}$
 on $ E_{q_k,q'_k}(\alpha,\beta)$ satisfies
$$ \Phi_{q_k,q'_k}
 :   n\alpha + m\beta \ \mapsto   \   \left| n - q_{k-1}\right|_{q_k}  \,  \alpha + \left|m + b'_k q'_{k-1}\right|_{q'_k} \, \beta, $$
  and
\begin{multline*}
\Delta_{q_k,q'_k} (n,m)  = \\
\begin{cases}
\| q_{k-1} \alpha \| - b'_k \| q'_{k-1} \beta \| ,  &q_{k-1} \le n, \   m < q'_k - b'_k q'_{k-1}, \\
\| q_{k-1} \alpha \| - b'_k \| q'_{k-1} \beta \| - \| q'_{k} \beta\|,   &q_{k-1} \le n, \ q'_k - b'_k q'_{k-1} \le m, \\ 
\| q_{k-1} \alpha \| - b'_k \| q'_{k-1} \beta \| + \| q_{k} \alpha\| ,  & n < q_{k-1}, \  m < q'_k - b'_k q'_{k-1}, \\
\| q_{k-1} \alpha \| - b'_k \| q'_{k-1} \beta \| + \| q_{k} \alpha\| -  \| q'_{k} \beta \|, & n < q_{k-1}, \  q'_k - b'_k q'_{k-1} \le m.
\end{cases}\end{multline*}
Let $ \varphi_{q_k,q'_k} $ be the map  defined on   ${\mathcal E}_{q_k,q'_k}$    as 
 $\varphi_{q_k,q'_k} (n,m) = ( | n - q_{k-1} |_{q_k}, | m + b'_k q'_{k-1} |_{q'_k} )$. Its action 
  is shown in Figure~\ref{fig:squarepoints2} (on the left) as an exchange of  4   sub-rectangles.
  Recall that  $\varphi_{q_k,q'_k}$ and $\Phi_{q_k,q'_k}$ are related by    \eqref{eq:varphi}.

\subsection*{Distribution of the points of $E_{q_k,N}(\alpha,\beta)$ for  $q_{k}  <  N  \leq  q'_{k}$} 
We obtain $\Phi_{q_k,N}$ on $E_{q_k,N}(\alpha,\beta)$  by iterating  the map $\Phi_{q_k,q'_k}$. Indeed, 
since  $E_{q_k,N}(\alpha,\beta)$ is a subset of  $E_{q_k,q'_k}(\alpha,\beta)$,
the neighbor map  $\Phi_{q_k,N}$ on $E_{q_k,N}(\alpha,\beta)$ is the induced map of $\Phi_{q_k,q'_k}$ on $E_{q_k,N}(\alpha,\beta)$, i.e.,
$$\Phi_{q_k,N} (x) = (\Phi_{q_k,q'_k})^{\tau(x)} (x),$$ 
where 
$\tau(x) = \min \{ \ell \ge 1 : (\Phi_{q_k,q'_k})^\ell (x) \in E_{q_k,N}(\alpha,\beta) \}$
is the first return  time to $ E_{q_k,N}(\alpha,\beta)$ of the map $\Phi_{q_k,q'_k}$. This  induction  step will create  two more sub-rectangles, that is,  $\varphi_{q_k,q'_k} $
acts as  an exchange of 4 sub-rectangles,  while $\varphi_{q_k,N} $
acts as an exchange of 6 sub-rectangles (see Figure~\ref{fig:squarepoints2}).

The following lemma   expresses the fact that the return time $\tau$ takes 3 values. Note that the statement  below does not  depend on   the parity of  $k$.

\begin{lemma}\label{lem:3G}  Let  $\tau$
be the first return  time to $ E_{q_k,N}(\alpha,\beta)$ of the map $\Phi_{q_k,q'_k}$.  There
 exist  $\tau_1, \tau_2, \tau_3$, $N_1,N_2,N_3$  such that, 
for each $n\alpha + m\beta \in E_{q_k,N}(\alpha,\beta)$,  we have 
\begin{equation*}
\tau (n\alpha+m\beta) = 
\begin{cases}
\tau_1 ,  & \mbox{ if }  \ 0  \le m < N_1, \\
\tau_1 + \tau_2,  & \mbox{ if } \  N_1 \le m < N_2, \\
\tau_2,  & \mbox{ if } \ N_2 \le m < N.
\end{cases}
\end{equation*}
Moreover, there exist $d_1,d_2$ nonnegative integers such that 
\begin{align*} 
[0, N_1) + \tau_1 b'_k q'_{k-1} &= [N-N_1, N) + d_1 q'_{k}, \\
[N_1, N_2) + (\tau_1+\tau_2) b'_k q'_{k-1} &= [N-N_2, N-N_1) + (d_1 +d_2) q'_{k}, \\
[N_2, N) + \tau_2 b'_k q'_{k-1} &= [0, N-N_2) + d_2 q'_{k}.
\end{align*}

\end{lemma}

\begin{proof} 
We prove  the lemma  for $k$ even, but  the  same argument works for  $k$ odd.
Recall that, for even $k$,
$\Phi_{q_k,q'_k} ( n\alpha + m\beta) =   | n -  q_{k-1} |_{q_k}  \,  \alpha + |m + b'_k q'_{k-1} |_{q'_k} \, \beta$.
Thus 
$$(\Phi_{q_k,q'_k})^\ell ( n\alpha + m\beta) \in   E_{q_k,N}(\alpha,\beta)\  \text{  if and only if } \ 
0 \le  | m + \ell  b'_kq'_{k-1} |_{q'_k} < N.$$

Let 
$$\bar \tau(m) = \min \{ \ell \ge 1 \, : \, 0 \le \, \left| m + \ell b'_kq'_{k-1} \right|_{q'_k} < N \}.$$
The discrete version of the three-gap problem (see  e.g. \cite{Slater:1967})  applied to the translation
by $b'_kq'_{k-1}$ modulo $N$ provides  the existence
of $\tau_1, \tau_2, \tau_3$, $N_1,N_2,N_3$ such that  \begin{equation*}
\bar \tau (m) = 
\begin{cases}
\tau_1 ,  & \mbox{ if }  0  \le m < N_1, \\
\tau_1 + \tau_2,  & \mbox{ if }  N_1 \le m < N_2, \\
\tau_2,  & \mbox{ if }  N_2 \le m < N,
\end{cases}
\end{equation*}
as well as the existence of
nonnegative integers $d_1, d_2$ satisfying 
\begin{align*}
[0, N_1) + \tau_1 b'_k q'_{k-1} &= [N-N_1, N) + d_1 q'_{k}, \\
[N_2, N) + \tau_2 b'_k q'_{k-1} &= [0, N-N_2) + d_2 q'_{k}.
\end{align*}
Clearly, we have
\begin{equation*}
[N_1, N_2) + (\tau_1+\tau_2) b'_k q'_{k-1} = [N-N_2, N-N_1) + (d_1 +d_2) q'_{k}.
\end{equation*}
Lemma~\ref{lem:3G} is  thus a direct consequence of the discrete three-gap problem.
\end{proof}

Therefore, for  $k$ even, we  deduce from Lemma~\ref{lem:3G} that
\begin{multline}\label{qkn}
\Phi_{q_k,N} (n\alpha+m\beta) = (\Phi_{q_k,q'_k})^{\tau(n\alpha+m\beta)} (n\alpha+m\beta) \\
=\begin{cases} 
\left | n - \tau_1 q_{k-1} \right |_{q_k} \alpha +  \left( m + \tau_1 b'_k q'_{k-1} - d_1 q'_{k} \right) \beta,
  &  0 \le m < N_1, \\
\left | n - (\tau_1+\tau_2) q_{k-1} \right |_{q_k} \alpha  \\
\quad +  \left( m + (\tau_1+\tau_2) b'_k q'_{k-1} - (d_1+d_2) q'_{k}  \right) \beta,
&   N_1 \le m < N_2, \\
\left | n - \tau_2 q_{k-1} \right|_{q_k} \alpha + \left( m + \tau_2 b'_k q'_{k-1} - d_2 q'_{k}  \right)  \beta,
  &  N_2 \le m < N.
\end{cases} 
\end{multline}
Let $h_1, h_2, h_3$ be nonnegative integers satisfying
$$ \tau_1 q_{k-1} = h_1 q_k + r_1, \quad  \tau_2 q_{k-1} = h_2 q_k + r_2, \quad 
(\tau_1 + \tau_2) q_{k-1} = h_3 q_k + r_3 $$ 
with $ 0 \le r_1, r_2, r_3 < q_k$.
Each of the three cases splits into two cases  according to  the fact  that $n$   is smaller  or  not than $r_i$, for  $i=1,2,3$.
Then, we have 
$$
\Delta_{q_k,N} (n,m) =
\begin{cases}
\Delta_1 +  \| q_{k} \alpha \|, &\text{ if } 0 \le n < r_1,\  0 \le m < N_1, \\
\Delta_1,  &\text{ if } r_1 \le n < q_k, \  0 \le m < N_1,\\
\Delta_3 +  \| q_{k} \alpha \|, &\text{ if }   0  \le n < r_3, \ N_1 \le m < N_2,\\
\Delta_3, &\text{ if } r_3 \le n < q_k, \ N_1 \le m < N_2,\\
\Delta_2 +  \| q_{k} \alpha \|,  &\text{ if } 0 \le n < r_2, \  N_2 \le m < N, \\
\Delta_2, &\text{ if } r_2 \le n < q_k, \  N_2 \le m < N,
\end{cases}
$$
where
\begin{align*}
\Delta_1 &= \tau_1 (\| q_{k-1}\alpha \| - b'_k \| q'_{k-1} \beta \|) - d_1 \| q'_{k} \beta \| + h_1 \| q_k \alpha\| , \\
\Delta_2 &= \tau_2 (\| q_{k-1}\alpha \| - b'_k \| q'_{k-1} \beta \|) - d_2 \| q'_{k} \beta \| + h_2 \| q_k \alpha \|, \\
\Delta_3 &= (\tau_1 + \tau_2) (\| q_{k-1}\alpha \| - b'_k \| q'_{k-1} \beta \|) - (d_1 +d_2) \| q'_{k} \beta \| + h_3 \| q_k \alpha \|.
\end{align*}
Indeed,  in the case where,  for example,   $0 \le m < N _1$, we have,  by \eqref{qkn},
\begin{align*}
&\Phi_{q_k,N} (n\alpha+m\beta) - (n\alpha + m\beta) \\
&\quad =  \left | n - \tau_1 q_{k-1} \right |_{q_k} \alpha +  \left( m + \tau_1 b'_k q'_{k-1} - d_1 q'_{k} \right) \beta - (n\alpha + m\beta) \\
&\quad = \left(  \left | n - r_1\right |_{q_k} - n \right) \alpha +  \left(\tau_1 b'_k q'_{k-1} - d_1 q'_{k} \right) \beta  \\
&\quad = \begin{cases} 
\left( - r_1 + q_k \right) \alpha +  \left(\tau_1 b'_k q'_{k-1} - d_1 q'_{k} \right) \beta , & \mbox{  if } n < r_1, \\
 - r_1 \alpha +  \left(\tau_1 b'_k q'_{k-1} - d_1 q'_{k} \right) \beta  , &  \mbox{  if }  n \ge r_1,
\end{cases} 
\end{align*}
and
\begin{multline*}
 - r_1 \alpha +  \left(\tau_1 b'_k q'_{k-1} - d_1 q'_{k} \right) \beta 
 = (h_1 q_k - \tau_1 q_{k-1} )\alpha + \left(\tau_1 b'_k q'_{k-1} - d_1 q'_{k} \right) \beta \\
=  \tau_1 (\| q_{k-1}\alpha \| - b'_k \| q'_{k-1} \beta \|) - d_1 \| q'_{k} \beta \| + h_1 \| q_k \alpha\| =\Delta_1. 
\end{multline*} 
The action of $\varphi_{q_k,N} (n,m) = ( \varphi_{q_k,q'_k} )^{ \tau (m)} (n,m) $ on ${\mathcal E}_{q_k,N}$
  is illustrated in Figure~\ref{fig:squarepoints2}.

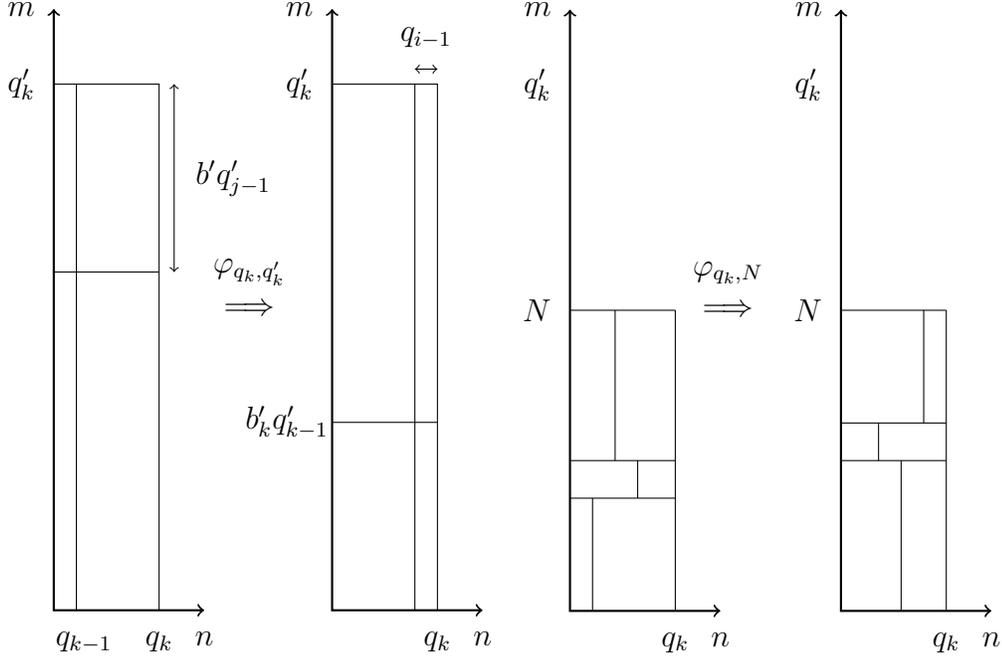
\begin{figure}
\begin{center}
\begin{tikzpicture}[every loop/.style={}]
  \draw [thick, <->] (2,0) -- (0,0) -- (0,8);
  \draw (.3,0)--(.3,7);
  \draw (0,4.5)--(1.4,4.5);
  \draw (1.4,0)--(1.4,7)--(0,7);
  \node [below=4pt] at (0.4,0) {$q_{k-1}$};   
  \node [below=4pt] at (1.4,0) {$q_k$};  
  \node [below=4pt] at (2,0) {$n$};  
  \node [left=3pt] at (0,7) {$q'_k$};  
  \node [left=3pt] at (0,8) {$m$};  
  \draw[<->] (1.6,4.5) to (1.6,7);
  \node [right=4pt] at (1.6,5.75) {$b' q'_{j-1}$}; 
  \node at (2.6,4) {$\Longrightarrow$};
  \node at (2.6,4.5) {$\varphi_{q_k,q'_k}$};
  \draw [thick, <->] (5.7,0) -- (3.7,0) -- (3.7,8);
  \draw (4.8,0)--(4.8,7);
  \draw (3.7,2.5)--(5.1,2.5);
  \draw (5.1,0)--(5.1,7)--(3.7,7);
  \node [below=4pt] at (5.1,0) {$q_k$};  
 \node [below=4pt] at (5.7,0) {$n$};  
 \node [left=3pt] at (3.9,2.5) {$b'_kq'_{k-1}$};
  \node [left=3pt] at (3.7,7) {$q'_k$};  
  \node [left=3pt] at (3.7,8) {$m$};  
  \draw[<->] (4.8,7.2) to (5.1,7.2);
  \node [above=4pt] at (4.95,7.2) {$q_{i-1}$}; 
\end{tikzpicture} 
\begin{tikzpicture}[every loop/.style={}]
  \draw [thick, <->] (2,0) -- (0,0) -- (0,8);
  \draw (1.4,0)--(1.4,4);
  \draw (0,4)--(1.4,4);
  \draw (0,1.5)--(1.4,1.5);
  \draw (0,2)--(1.4,2);
  \draw (.3,0)--(.3,1.5);
  \draw (.6,2)--(.6,4);
  \draw (.9,1.5)--(.9,2);
  \node [below=4pt] at (1.4,0) {$q_k$};  
  \node [below=4pt] at (2,0) {$n$};  
  \node [left=3pt] at (0,7) {$q'_k$};  
  \node [left=3pt] at (0,8) {$m$};  
  \node [left=3pt] at (0,4) {$N$};
  \node at (2.1,4) {$\Longrightarrow$};
  \node at (2.1,4.5) {$\varphi_{q_k,N}$};
\end{tikzpicture}
\begin{tikzpicture}[every loop/.style={}]
  \draw [thick, <->] (2,0) -- (0,0) -- (0,8);
  \draw (0,2.5)--(1.4,2.5);
  \draw (1.4,0)--(1.4,4);
  \draw (0,4)--(1.4,4);
  \draw (0,2)--(1.4,2);
  \draw (1.1,2.5)--(1.1,4);
  \draw (.8,0)--(.8,2);
  \draw (.5,2)--(.5,2.5);
  \node [below=4pt] at (1.4,0) {$q_k$};  
  \node [below=4pt] at (2,0) {$n$};  
  \node [left=3pt] at (0,7) {$q'_k$};  
  \node [left=3pt] at (0,8) {$m$};  
  \node [left=3pt] at (0,4) {$N$};
\end{tikzpicture} 
\end{center}
\caption{The action  of $\varphi_{q_k,q'_k}$ on ${\mathcal E}_{q_k,q'_k} (\alpha,\beta)$ is  an exchange  of 4 sub-rectangles (left),
and the action  of $\varphi_{q_k,N}$ on ${\mathcal E}_{q_k,N} (\alpha,\beta)$ is an exchange of  6 sub-rectangles (right).}
\label{fig:squarepoints2}
\end{figure}

\subsection*{From $E_{q_k,N}(\alpha,\beta)$ to $E_{N}(\alpha,\beta)$} 
Let $N = a q_k + R$,  with   $a \geq 1$ and $1 \le R \le q_k$
(recall that $q_k <  N \le q'_{k}$). 
Since $E_{q_k,N}$ is a subset of  $E_{q_k, q'_k}$, we have 
$$\min \Delta \left(E_{q_k,N} 
 \right) \ge 
 \min \Delta \left( E_{q_k, q'_k} 
  \right)= \| q_{k-1} \alpha\| - b'_k \| q'_{k-1}\beta \| - \| q'_{k} \beta\|. $$
Using \eqref{33} and \eqref{eqq},
it follows that
\begin{equation}\label{lowerbound}
\begin{split}
\min \Delta \left(E_{q_k,N} 
 \right)
&\ge  \| q_{k-1} \alpha\| - b'_k \| q'_{k-1}\beta \| - \| q'_{k} \beta\|  \\
& 
>  \Big( (q_k)^5 -\frac{1}{q_k} \Big)\| q_k \alpha\|  - \frac{\| q_k \alpha\| }{b'_{k+1}} 
> ( (q_k)^3 +1 ) \| q_k \alpha \| \\
&>\frac{b'_{k}q_k + 1}{q_k} \| q_k \alpha \| 
= \frac{ q'_{k} }{q_k} \| q_k \alpha \|
\ge \frac{N}{q_k} \| q_k \alpha \| >  a \| q_k \alpha \| .
\end{split}\end{equation}

We claim that
\begin{equation}\label{lem:extension}
\Phi_{N}(n \alpha + m\beta) = \begin{cases}
(n + q_k) \alpha + m \beta, & \text{ if } 0  \le n < N - q_k, \\ 
\Phi_{q_k,N} \left(  \left| n \right|_{q_k} \alpha + m \beta \right), & \text{ if } N - q_k \le n < N.
\end{cases}
\end{equation}

\begin{proof}[Proof of  Claim \eqref{lem:extension}]
If $0 \le n < R=N-aq_k$, then  the points
$(n + q_k) \alpha + m \beta$, $( n + 2q_k) \alpha + m \beta$, \dots,  $( n + a q_k) \alpha + m\beta$ are located between $n\alpha + m \beta$ and $\Phi_{q_k,N} (n\alpha + m \beta)$, as shown in  Figure~\ref{fig1}.
Therefore, we have $\Delta_{N}(n + cq_k,m) = \| q_k \alpha \|$, for $0 \le c \le a-1$, and 
$\Delta_{N}(n + a q_k,m) = \Delta_{q_k,N}(n,m) - a \| q_k \alpha \|$, which is positive by \eqref{lowerbound}.

If $R \le n < q_k$, then the points  $(n+q_k)\alpha + m \beta$, $(n+ 2q_k)\alpha + m\beta$, \dots,  $(n+ (a-1)q_k)\alpha + m\beta$ are located  between $n\alpha + m \beta$ and $\Phi_{q_k,N} (n\alpha + m \beta)$. 
In this case, the gaps between two adjacent points of $E_{N}(\alpha,\beta)$ are 
given by $\Delta_{N} (n + cq_k,m) = \| q_k \alpha \|$, for $0 \le c \le a-2$, and 
$\Delta_{N}(n + (a-1) q_k,m) = \Delta_{q_k,N}(n,m) - (a-1) \| q_k \alpha \|$, which is positive by \eqref{lowerbound}.
\end{proof}

\begin{figure}
\begin{center}
\begin{tikzpicture}[every loop/.style={}]
 \node at (0,0) {$\bullet$} node [below=5pt] at (0,0) {$n\alpha + m \beta$};
 \node at (1.5,0) {$\bullet$} node [above=5pt] at (1.5,0) {$(n + q_k)\alpha + m \beta$};
 \node at (3,0) {$\bullet$} node [below=5pt] at (3,0) {$(n + 2q_k)\alpha + m \beta$};
 \node at (6.7,0) {$\bullet$} node [below=5pt] at (4.9,0) {$\cdots$} node [below=5pt] at (6.7,0) {$(n + a q_k)\alpha + m \beta$};
 \node at (11,0) {$\bullet$}  node [below=5pt] at (11,0) {$\Phi_{q_k,N} (n\alpha + m \beta)$}; 
 \draw  (-.5,0) -- (11.5,0);
\end{tikzpicture} 
\end{center}
\caption{Illustration of the  proof of Claim \eqref{lem:extension}, when $0 \le n < R$.}
\label{fig1}
\end{figure}
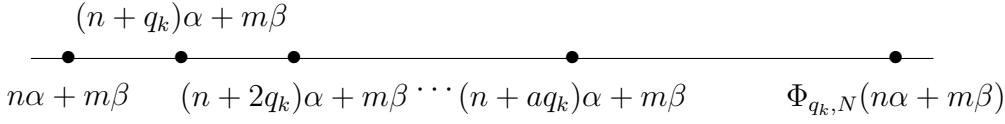

Therefore, using \eqref{qkn}, we deduce 
that
\begin{multline*}
\Phi_{N} (n\alpha+m\beta) \\
=  \left\{
\begin{aligned}
&(n + q_k)\alpha + m\beta, && 0 \le n < N- q_k, \\
&\left | n - \tau_1 q_{k-1} \right |_{q_k} \alpha+ \left( m + \tau_1 b'_k q'_{k-1} - d_1 q'_{k} \right) \beta, \! && n \ge N- q_k, 0 \le m < N_1, \\
& \left | n - (\tau_1+\tau_2) q_{k-1} \right |_{q_k} \alpha  \\
&\ + \left( m + (\tau_1+\tau_2) b'_k q'_{k-1} - (d_1+d_2) q'_{k}  \right) \beta, \!  && n \ge N- q_k, N_1 \le m < N_2, \\
& \left | n - \tau_2 q_{k-1} \right|_{q_k} \alpha + \left( m + \tau_2 b'_k q'_{k-1} - d_2 q'_{k}  \right)  \beta, \! && n \ge N- q_k, N_2 \le m < N.
\end{aligned}
\right.
\end{multline*}
Let $\bar h_1, \bar h_2, \bar h_3$ be nonnegative integers satisfying
$$ N- \tau_1 q_{k-1} = \bar h_1 q_k + \bar r_1, \  
N - \tau_2 q_{k-1} = \bar h_2 q_k + \bar r_2, \ 
N - (\tau_1 + \tau_2) q_{k-1} = \bar h_3 q_k + \bar r_3 $$
 with $ 0 \le\bar r_1,\bar r_2, \bar r_3 < q_k. $
Then, we have 
$$\Delta_{N} (n,m) =
\begin{cases}
\| q_k \alpha \| , & 0 \le n < N-q_k, \\
\Delta_1 +  \| q_{k} \alpha \|, & N-q_k \le n < N- \bar r_1,\  0 \le m < N_1, \\
\Delta_1,  & N- \bar r_1 \le n < N, \  0 \le m < N_1,\\
\Delta_3 +  \| q_{k} \alpha \|, & N-q_k \le n < N - r_3, \ N_1 \le m < N_2,\\
\Delta_3, & N - \bar r_3 \le n < q_k, \ N_1 \le m < N_2,\\
\Delta_2 +  \| q_{k} \alpha \|,  & N - q_k \le n < N - \bar r_2, \  N_2 \le m < N, \\
\Delta_2, & N - \bar r_2 \le n < N, \  N_2 \le m < N,
\end{cases}
$$
where
\begin{align*}
\Delta_1 &= \tau_1 (\| q_{k-1}\alpha \| - b'_k \| q'_{k-1} \beta \|) - d_1 \| q'_{k} \beta \| - \bar h_1 \| q_k \alpha\| , \\
\Delta_2 &= \tau_2 (\| q_{k-1}\alpha \| - b'_k \| q'_{k-1} \beta \|) - d_2 \| q'_{k} \beta \| - \bar h_2 \| q_k \alpha\|, \\
\Delta_3 &= (\tau_1 + \tau_2) (\| q_{k-1}\alpha \| - b'_k \| q'_{k-1} \beta \|) - (d_1 +d_2) \| q'_{k} \beta \| - \bar h_3\| q_k \alpha \|.
\end{align*}
 The action of $\varphi_{N} (n,m)$ on ${\mathcal E}_{N}$ is illustrated in Figure~\ref{fig:squarepoints3}.

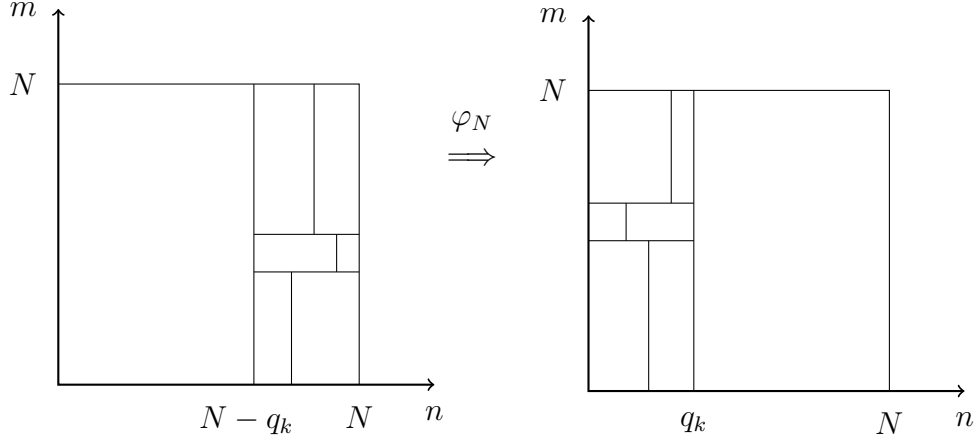
\begin{figure}
\begin{center} 
\begin{tikzpicture}[every loop/.style={}]
  \draw [thick, <->] (5,0) -- (0,0) -- (0,5);
  \draw (2.6,0)--(2.6,4);
  \draw (0,4)--(4,4)--(4,0);
  \draw (2.6,1.5)--(4,1.5);
  \draw (2.6,2)--(4,2);
  \draw (3.1,0)--(3.1,1.5);
  \draw (3.4,2)--(3.4,4);
  \draw (3.7,1.5)--(3.7,2);
  \node [below=4pt] at (5,0) {$n$};  
  \node [left=4pt] at (0,5) {$m$};  
  \node [below=4pt] at (4,0) {$N$};  
   \node [below=4pt] at (2.5,0) {$N-q_k$}; 
  \node [left=4pt] at (0,4) {$N$};
  \node at (5.5,3) {$\Longrightarrow$};
  \node at (5.5,3.5) {$\varphi_{N}$};
\end{tikzpicture} \
\begin{tikzpicture}[every loop/.style={}]
  \draw [thick, <->] (5,0) -- (0,0) -- (0,5);
  \draw (0,2.5)--(1.4,2.5);
  \draw (1.4,0)--(1.4,4);
  \draw (0,4)--(4,4)--(4,0);
  \draw (0,2)--(1.4,2);
  \draw (1.1,2.5)--(1.1,4);
  \draw (.8,0)--(.8,2);
  \draw (.5,2)--(.5,2.5);
  \node [below=4pt] at (1.4,0) {$q_k$};  
  \node [below=4pt] at (5,0) {$n$};  
  \node [left=4pt] at (0,5) {$m$};  
 \node [below=4pt] at (4,0) {$N$};  
  \node [left=4pt] at (0,4) {$N$};
\end{tikzpicture}
\end{center}
\caption{The action of  $\varphi_{N}$ on ${\mathcal E}_{N}$  is  an  exchange of  7 sub-rectangles.}
\label{fig:squarepoints3}
\end{figure}

\subsection*{End of the proof} The case $q_k < N \leq q'_{k}$ and $k$ even  has thus been handled.
In the case  $ k$   odd, $\Delta_{q_k,q'_k} (n,m)$ still  takes four values, as discussed in Lemma~\ref{lem:exchange}, namely 
\begin{align*}
&\| q_{k-1}\alpha \| - b'_k \| q_{k-1} \beta \| , &
&\| q_{k-1}\alpha \| - b'_k \| q_{k-1} \beta \| - \| q'_{k} \beta\|, \\ 
&\| q_{k-1}\alpha \| - b'_k \| q_{k-1} \beta \| + \| q_{k} \alpha\| , &
&\| q_{k-1}\alpha \| - b'_k \| q_{k-1} \beta \| + \| q_{k} \alpha\| -  \| q'_{k} \beta \|.
\end{align*}
We also have at most 6 values for
$\Delta \left(E_{q_k,N} (\alpha, \beta)\right) $ which are obtained by considering the  induced map of $\Phi_{q_k,q'_k}$.
Observe that $\Delta_{q_k,q'_k} (n,m)$ takes the  same values as  in  the case $q_k < N \leq q'_{k}$ and $k$ even. It follows
from Claim~\eqref{lem:extension}    that there are 7 values for $\Delta \left(E_{N} (\alpha, \beta)\right) $.

Lastly, the case  $q'_{k}< N \leq q_{k+1}$  is  similarly  deduced by induction  from the case    $E_{q'_k,q_{k+1}} (\alpha, \beta)$.
This ends the proof of Theorem~\ref{thm} (i).

\begin{remark}\label{rem:primitivebis}
Observe that, in  continuation of  Remark  \ref{rem:primitive},  there are  here also    4  primitive lengths.
\end{remark}

\section{Unbounded  number of lengths}\label{sec:unbounded}

This section is devoted to the proof of Statement (ii) of Theorem \ref{thm}. 
The strategy works as follows: one wants to regularly get indices  $k$  for which 
$q_k = q'_k +1$. This will imply that $q_k$ and $q'_k$ are coprime and   that they have the same size.
This will allow us   in particular  to consider mainly   the first level $E_{q_k,q'_k}(\alpha,\beta)$ (and in fact even  here  $E_{q_{4k+1},q'_{4k+1}}(\alpha,\beta))$). 
Now, we  provide a construction of $\alpha$ and $\beta$ for which  $ q_{4k-3} + 1 = q'_{4k-3}$  and $ q_{4k-2} - 1 = q'_{4k-2}$, for all $k$.
Furthermore,  we will have    $a'_{4k+1}=1$.
 
 \subsection*{Construction of the sequences of  convergents $(q_i)_i$ and $(q'_j)_j$.} 
We consider $\alpha$ and $\beta$  irrationals in $(0,1)$ with  respective sequences of partial quotients $(a_i)_i$  and $(a'_j)_j$ satisfying 
\begin{equation*}
a_1 =  2, \quad  a_2 =  2 \quad \text{ and } \quad  a'_1 =  3, \quad  a'_2 =  1.
\end{equation*}
Then we have
\begin{equation*}
q_0 =  1, \ q_1 =  2, \ q_2 =  5, \quad
q'_0 =  1, \ q'_1 =  3, \ q'_2 =  4.
\end{equation*}

We  now inductively define $(a_i)_i, (a'_j)_j$ as follows.
Suppose that
$$ q_{4k-3} + 1 = q'_{4k-3} , \qquad  q_{4k-2} -1 = q'_{4k-2}   \qquad \mbox{for } k \geq 1. $$
Furthermore, let $R_k:=   q_{4k-3} + 1 = q'_{4k-3} $ and $Q_{k}:= q_{4k-2} -1 = q'_{4k-2}.$

Let
\begin{gather*}
a_{4k-1} :=  1,  \quad
a_{4k} :=  3,   \quad
a_{4k+1} :=  2Q_k +R_k-1,  \quad
a_{4k+2} :=  6Q_k+4R_k, \\
a'_{4k-1} :=  2,\quad
a'_{4k} :=  4 Q_k+3R_k -2,\quad
a'_{4k+1} := 1, \quad
a'_{4k+2} := 6Q_k+4R_k -1.
\end{gather*}
Then we get 
\begin{align*}
q_{4k-1} &= Q_k+R_k, & q'_{4k-1} &=  2Q_k+R_k,\\
q_{4k} &=  4Q_k+3R_k+1, & q'_{4k} &=  8Q_k^2+(10R_k-3)Q_k + 3R_k^2-2R_k,
\end{align*}
and 
$$ q_{4k+1}=R_{k+1} - 1 ,  \quad  q'_{4k+1}=R_{k+1} , \quad  q_{4k+2} = Q_{k+1} +1, \quad q'_{4k+2} =Q_{k+1}, $$
where we put inductively
\begin{align*}
R_{k+1} &= 8Q_k^2+(10R_k-1)Q_k + 3R_k^2 -R_k,\\
Q_{k+1} &= 48Q_k^3 + (96R_k-6)Q_k^2 + (43R_k^2-5R_k-2)Q_k + 12R_k^3-4R_k^2 -R_k.
\end{align*}
 


\subsection*{Rational independence of $1,\alpha,\beta$}

Suppose that $1,\alpha,\beta$ are rationally dependent.
Then, there exist integers $n_0, n_1, n_2$ satisfying
$n_0 + n_1 \alpha + n_2 \beta = 0$. Since $\alpha, \beta$ are both
irrational numbers,
$n_1, n_2 \ne 0$.
Then we have
$$
\| n_1 q'_{4k+1} \alpha \| = \| n_2 q'_{4k+1} \beta \| \le |n_2| \|
q'_{4k+1} \beta \| < \frac {|n_2|}{q'_{4k+2}}< \frac {|n_2|}{a'_{4k+2}
q'_{4k+1}}.$$
Thus, there exists an integer $p$ satisfying
\begin{equation}\label{eq:indep}
\left | \alpha  - \frac{p}{n_1 q'_{4k+1} } \right | < \frac
{|n_2|}{|n_1| a'_{4k+2} (q'_{4k+1})^2}.
\end{equation}
Choose $k$ large enough for $a'_{4k+2} > 2|n_1||n_2|$ to hold.
Then, by Legendre's theorem (see e.g.  \cite[Theorem 1.8]{Bugeaud}), one gets
$\frac{p}{n_1 q'_{4k+1}} = \frac{p_s}{q_s}$ for some positive integer $s$.
Since $q_{4k+2} = q'_{4k+2} +1 > a'_{4k+2} q'_{4k+1} > |n_1| q'_{4k+1}$,  we get $s \le 4k+1$.
Also from $q'_{4k+1} = q_{4k+1}+1$,  we get $s \ne 4k+1$.
If we assume $s \le 4k$, then
$$
\left | \alpha - \frac{p_s}{q_s} \right | \ge \left | \alpha -
\frac{p_{4k}}{q_{4k}} \right | \geq 
\frac {1}{2q_{4k} q_{4k+1}} > \frac {1}{2 (q'_{4k+1})^2 } > \frac{|n_1||n_2|}{a'_{4k+2} (q'_{4k+1})^2},$$
which is a contradiction to \eqref{eq:indep}.

\subsection*{Organization of the proof} 

We will work mainly with the points   of the  first level provided by $E_{q_{4k+1},q'_{4k+1}} (\alpha, \beta)$. This will  be sufficient  to derive   infinitely many lengths
 for the  points in  $E_N$, with $N= q_{4k+1}=q'_{4k+1} -1$.  

The  study of the  first level will   be  divided into  Lemma  \ref{lem:positive}  and Proposition \ref{prop:unbounded}.  The main difficulty here is 
that   the map $\phi$ of   Lemma \ref{lem:exchange}   provides points that can  be  located either on the right, or on the  left of   a point (Assumption  \eqref{eq:assumption}
does not hold). 

\subsection*{Distribution of the points of $E_{q_{4k+1},q'_{4k+1}} (\alpha, \beta)$.} We now consider points of the  first level provided by $E_{q_{4k+1},q'_{4k+1}} (\alpha, \beta)$.
Recall that $q'_{4k+1}= q_{4k+1}+ 1$.
With the notation of  Lemma \ref{lem:exchange},  put $i = j =4k+1$.
Observe  that $b'=1$. We  consider 
$$ \delta_k : =  \| q'_{4k-1} \beta\| - \| q_{4k} \alpha \|.$$
According to Lemma \ref{lem:positive} below, one has $\delta_k>0$. 

Note that there are  more than the 4  lengths of Lemma \ref{lem:exchange} since   Assumption (\ref{eq:assumption})  is not satisfied.
Indeed, one has $ - \| q_{4k}\alpha \| + \| q'_{4k-1}\beta \|  = \delta_k >0, $ 
 which  contradicts  $\| q'_{4k+1} \beta\| <   \| q_{4k}\alpha \| -\| q'_{4k}\beta \| $, by noticing that   $\| q'_{4k-1}\beta \|= \| q'_{4k}\beta \| +\| q'_{4k+1}\beta \|$, since $a'_{4k+1}=1$.
However, even if Assumption  (\ref{eq:assumption})  is not satisfied, Lemma   \ref{lem:exchange} provides  a convenient   expression $\phi$  for the neighbor map,
that will be used in the proof of Proposition  \ref{prop:unbounded} below, showing that there are  at most 12  lengths.
 
\begin{lemma}\label{lem:positive}
One has,  for all $k$:
$$
0 < 2 \delta_k a_{4k+1}  <\| q_{4k+1} \alpha \| < \| q'_{4k+1} \beta \|.
$$
\end{lemma} 

\begin{proof}

We  have 
$$
\| q_{k-1} \alpha \| = \frac{1}{q_{k}+q_{k-1}\frac{\| q_{k} \alpha \|}{\| q_{k-1} \alpha \|}} 
= \frac{1}{q_{k}+\cfrac{q_{k-1}}{ a_{k+1} + \cfrac{1}{a_{k+2} + \ddots}}}. 
$$

Hence  we get  
\begin{align*}
\| q_{4k} \alpha \|
&= \frac{1}{R_{k+1}- 1 +\frac{4Q_k+3R_k+1}{ 6Q_k+4R_k + s }} 
= \frac{1}{ R_{k+1} - \frac 13 + \frac{ R_k + 3 -2 s }{3(6Q_k+ 4R_k+s)}}, \\
\| q'_{4k-1} \beta \| 
&=\frac{1}{ R_{k+1} - 2Q_k - R_k +\frac{2Q_k+R_k}{1+ \frac{1}{6Q_k+4R_k-1+ s'}}} 
=  \frac{1}{ R_{k+1} - \frac 13 + \frac{R_k +s'}{3(6Q_k+ 4R_k+s')}},
\end{align*}
where 
\begin{equation*}
s := \frac{1}{a_{4k+3} +\dfrac{1}{a_{4k+4}+ \ddots}}, \qquad 
s'  := \frac{1}{a'_{4k+3} +\dfrac{1}{a'_{4k+4}+ \ddots}}
\end{equation*}
satisfying that 
\begin{align*}
 \frac{6Q_{k+1} + 3R_{k+1} +1 }{8Q_{k+1}  + 4R_{k+1} +1 } = \frac{3a_{4k+5}+4}{4a_{4k+5}+5} < s & < \frac{3a_{4k+5}+1}{4a_{4k+5}+1} = \frac{6Q_{k+1} + 3R_{k+1} -2}{8Q_{k+1}  + 4R_{k+1} -3}, \\
\frac{4 Q_{k+1}  + 3R_{k+1} -2}{8Q_{k+1}  + 6R_{k+1} -3} = \frac{a'_{4k+4}}{2 a'_{4k+4} + 1} < s'& < \frac{a'_{4k+4} + 1}{2 a'_{4k+4} + 3} = \frac{4 Q_{k+1}  + 3R_{k+1} -1}{8Q_{k+1}  + 6R_{k+1} -1}.
\end{align*}
Then, we have 
\begin{align*}
\delta_k  &= \frac{\frac{ R_k + 3-2s}{3(6Q_k+ 4R_k+s)} - \frac{R_k + s'}{3(6Q_k+ 4R_k+s')} }
{\Big( R_{k+1} - \frac 13 + \frac{R_k + s'}{3(6Q_k+ 4R_k+ s')} \Big)  \Big( R_{k+1} - \frac 13 + \frac{ R_k + 3-2s}{3(6Q_k+ 4R_k+ s)} \Big)}.
\end{align*}  
By elementary computation we get
\begin{equation*}
 \frac{1}{3 (6Q_k + 5R_k)R_{k+1}^2} < \delta_k <  \frac{1}{3 (6Q_k + 4R_k)(R_{k+1} - \frac 13 )^2}.
\end{equation*}

Also we have
\begin{equation*}
\| q_{4k+1} \alpha \|
= \frac{1}{ Q_{k+1} +1 + (R_{k+1} - 1) s },  \qquad
\| q'_{4k+1} \beta \|
= \frac{1}{ Q_{k+1} + R_{k+1} s'},
\end{equation*}
thus
\begin{align*}
0<2 \delta_k a_{4k+1} < \frac{1}{3(R_{k+1}-\frac13)^2} &< \frac{1}{Q_{k+1}+ R_{k+1}} \\
&<\| q_{4k+1} \alpha \| < \| q'_{4k+1} \beta \| <\frac{1}{Q_{k+1}}.
\end{align*}
\end{proof}

\begin{proposition}\label{prop:unbounded}
Let $\alpha, \beta$ given by the construction  above. 
We consider points of the  first level provided by $E_{q_{4k+1},q'_{4k+1}} (\alpha,\beta)$.
The neighbor  map  $\Phi= \Phi_{q_{4k+1},q'_{4k+1}} $    satisfies the following.
\begin{enumerate}
\item If $q_{4k} \le n < q_{4k+1}, \ q'_{4k-1} \le m < q'_{4k+1}$, then
$$\Phi(n\alpha + m\beta) = (n - q_{4k} ) \alpha + (m -q'_{4k-1}) \beta, \quad \mbox{ and }  \quad   \Delta(n,m)=\delta_k.$$

\item If  $ 0 \le n < q_{4k-1}, \ q'_{4k+1} - a_{4k+1}q'_{4k-1} \le m < q'_{4k+1}$, then
\begin{align*}
\Phi(n\alpha + m\beta) &= (n + a_{4k+1} q_{4k} ) \alpha + ( m + a_{4k+1} q'_{4k-1} - q'_{4k+1}) \beta,  \\
\Delta(n,m)&= \| q'_{4k+1} \beta \|-a_{4k+1}\delta_k .
\end{align*}
\item If  $ q_{4k-1} \le n < q_{4k}, \ q'_{4k+1} - (a_{4k+1}-1)q'_{4k-1} \le m < q'_{4k+1}$, then
\begin{align*}
\Phi(n\alpha + m\beta) &= (n + (a_{4k+1} -1) q_{4k} ) \alpha + ( m + (a_{4k+1}-1) q'_{4k-1} - q'_{4k+1}) \beta, \\
 \Delta(n,m)&=\| q'_{4k+1} \beta \|-(a_{4k+1}-1)\delta_k.
\end{align*}
\item If  $q_{4k-1} \le n < 2q_{4k-1}, \ q'_{4k+1} - a_{4k+1}q'_{4k-1} \le m < q'_{4k+1} - (a_{4k+1}-1)q'_{4k-1}$, then
\begin{align*}
\Phi(n\alpha + m\beta) &= (n + 2a_{4k+1} q_{4k} -q_{4k+1}) \alpha + ( m + 2a_{4k+1} q'_{4k-1} - q'_{4k+1}) \beta, \\  
\Delta(n,m)&= \| q_{4k+1} \alpha \| +  \| q'_{4k+1} \beta \|  - 2a_{4k+1} \delta_k.
\end{align*}
\item If  $2q_{4k-1} \le n < q_{4k}, \ q'_{4k+1} - a_{4k+1}q'_{4k-1} \le m < q'_{4k+1} - (a_{4k+1}-1)q'_{4k-1}$, then
\begin{align*}
\Phi(n\alpha + m\beta) &= (n + (2a_{4k+1} -1) q_{4k} -q_{4k+1}) \alpha \\
&\quad + ( m + (2a_{4k+1} -1)q'_{4k-1} - q'_{4k+1}) \beta, \\  
\Delta(n,m)&=  \| q_{4k+1} \alpha \| +  \| q'_{4k+1} \beta \|  - (2a_{4k+1} -1) \delta_k.
\end{align*}
\item  If $0 \le n < q_{4k-1}, \ q'_{4k+1} - (a_{4k+1}+1)q'_{4k-1} \le m < q'_{4k+1} - a_{4k+1}q'_{4k-1}$, then
\begin{align*}
\Phi(n\alpha + m\beta) &= (n + 2a_{4k+1} q_{4k} -q_{4k+1}) \alpha + ( m + 2a_{4k+1} q'_{4k-1} - q'_{4k+1}) \beta,\\
 \Delta(n,m)&=  \| q_{4k+1} \alpha \| +  \| q'_{4k+1} \beta \|  - 2a_{4k+1} \delta_k.
\end{align*}
\item If $ q_{4k-1} \le n < q_{4k}, \ q'_{4k+1}- (a_{4k+1}+1)q'_{4k-1} \le m < q'_{4k+1}- a_{4k+1}q'_{4k-1}$, then
\begin{align*}
\Phi(n\alpha + m\beta) &= (n + a_{4k+1} q_{4k} - q_{4k+1}) \alpha + ( m + a_{4k+1}q'_{4k-1} ) \beta, \\ 
 \Delta(n,m)&=  \| q_{4k+1} \alpha \|  - a_{4k+1} \delta_k.
\end{align*}
\item If $ 0 \le n < q_{4k}, \ q'_{4k+1}- (a_{4k+1}+c+1)q'_{4k-1} \le m < q'_{4k+1}- (a_{4k+1}+c) q'_{4k-1}$, where $1 \le c \le a_{4k+1}-1$, then
\begin{align*}
\Phi(n\alpha + m\beta) &= (n + (a_{4k+1} +c)q_{4k} - q_{4k+1}) \alpha + ( m + (a_{4k+1}+c)q'_{4k-1} ) \beta, \\
 \Delta(n,m)&=  \| q_{4k+1} \alpha \|  - ( a_{4k+1} +c) \delta_k.
\end{align*}
\item If $ 0 \le n < 2q_{4k-1}, \ 0 \le m < q'_{4k+1}- 2a_{4k+1} q'_{4k-1} $, then
\begin{align*}
\Phi(n\alpha + m\beta) &= (n + 2a_{4k+1} q_{4k} - q_{4k+1}) \alpha + ( m + 2a_{4k+1}q'_{4k-1}) \beta, \\
 \Delta(n,m) &=   \| q_{4k+1} \alpha \|  - 2a_{4k+1} \delta_k.
\end{align*}
\item If  $ 2q_{4k-1} \le n < q_{4k}, \ 0 \le m < q'_{4k+1}- (2a_{4k+1} -1)q'_{4k-1}$, then
\begin{align*}
\Phi(n\alpha + m\beta) &= (n + (2a_{4k+1}-1) q_{4k} - q_{4k+1}) \alpha + ( m + (2a_{4k+1}-1) q'_{4k-1} ) \beta, \\  
\Delta(n,m) &=   \| q_{4k+1} \alpha \|  - ( 2a_{4k+1} -1) \delta_k.
\end{align*}
\item If  $(c-1)q_{4k}+ 2q_{4k-1} \le n < cq_{4k} +2q_{4k-1}, \ 0 \le m < q'_{4k-1}$, where $1 \le c \le a_{4k+1}-1$, then
\begin{align*}
\Phi(n\alpha + m\beta) &= (n + (2a_{4k+1}-c) q_{4k} - q_{4k+1}) \alpha + ( m + (2a_{4k+1}-c)q'_{4k-1} ) \beta, \\ 
 \Delta(n,m) &=   \| q_{4k+1} \alpha \|  - ( 2a_{4k+1} -c) \delta_k.
\end{align*}
\item If  $(a_{4k+1}-1)q_{4k}+ 2q_{4k-1} \le n < q_{4k+1}, \ 0 \le m < q'_{4k-1} $, then
\begin{align*}
\Phi(n\alpha + m\beta) &= (n + a_{4k+1}q_{4k} - q_{4k+1}) \alpha + ( m + a_{4k+1}q'_{4k-1} ) \beta, \\ 
 \Delta(n,m)&=  \| q_{4k+1} \alpha \|  - a_{4k+1} \delta_k.
\end{align*}
\end{enumerate}

\end{proposition}

\begin{remark}\label{rmk:overlap}
Observe that there are overlapped regions between  points corresponding to Case (8) and Case  (10)  (when $c=a_{4k+1}-1$), and between  Case (10) and Case  (11) (when $c=1$).
\end{remark}

\begin{proof}

Let $\overline{\Phi}$ and $\overline{\Delta}$  stand for the functions   defined  in the statement of the proposition.
We want to prove that  $\overline{\Phi}$ coincides with $\Phi$  on  $E_{q_{4k+1},q'_{4k+1}}$, and that   similarly $\overline{\Delta}$ coincides with $\Delta$.
The  proof works as  for Lemma  \ref{lem:exchange}: we will  show that the sum of the  lengths provided  by $\overline{\Delta}$  (with multiplicities)   equals  $1$.

According to Lemma \ref{lem:positive}, one checks that the lengths $\overline{\Delta}$  are all nonnegative.
The intervals  for the pairs $ (n,m) $  in Proposition \ref{prop:unbounded} are  also well-defined. Indeed,  one   checks that 
$q_{4k}- 2  q_{4k-1} >0$, by  noticing  that $a_{4k} = 3$, and  also   $q'_{4k+1} - (2 a_{4k+1} +1) q'_{4k-1}  >0$.

One has  $q'_{4k+1} =  q_{4k+1} +1$.   With the notation of Lemma \ref{lem:exchange}, one has   $b'=1$.
Since $ q'_{4k-1} = q'_{4k+1} - a'_{4k+1} q'_{4k} = q'_{4k+1} - q'_{4k}$,
we have
$$  | m - q'_{4k} |_{q'_{4k+1}} = | m + q'_{4k-1} |_{q'_{4k+1}} .$$
As in  Lemma \ref{lem:exchange},   we consider  the cyclic permutation  $\phi$ on $E_{q_{4k+1},q'_{4k+1}}(\alpha,\beta)$ defined, for all $(n,m) \in {\mathcal E} _{q_{4k+1},q'_{4k+1}} $, as: 
\begin{align*}
\phi (n \alpha + m \beta) 
&= |n+  q_{4k}|_{q_{4k+1}} \,  \alpha +  | m - {q'_{4k}}|_{q'_{4k+1}}\, \beta \\
&=  | n+ q_{4k} |_{q_{4k+1}}  \,  \alpha + |m + q'_{4k-1} |_{q'_{4k+1}} \, \beta.
\end{align*}
We   first   provide some  dynamical insight     on the way  the 12  lengths in Proposition  \ref{prop:unbounded}  have been obtained. As stressed before, Assumption \eqref{eq:assumption} is not satisfied, and   the  neighbor map $\Phi$ is not equal
to  $\phi$.    In fact,  there  are points $x$ for which  $\phi(x) $  is obtained from $x$  by  performing  a clockwise jump of $\delta_k$, but there are also  points $x$  for which  $\phi(x)$ is  located in   the  anticlockwise direction $(\phi(x)<x$), with  $x$ being the clockwise neighbor of $\phi(x)$.   However,  the map $\Phi$ can ce be recovered by  performing suitable inductions of  the map   $\phi$  on the set of points for which $ \phi(x)>x$.
Let 
$$ G := \{ n\alpha + m\beta : q_{4k+1} - q_{4k} \le n < q_{4k+1}  \text{ or }  q'_{4k+1} - q'_{4k-1} =q'_{4k} \le m < q'_{4k+1} \}.$$
One has $G \subset E_{q_{4k+1},q'_{4k+1}} (\alpha, \beta).$
Then $G$ is the set of points such that $\phi(x) >x $, that is,   $\phi (x) $  is obtained from $ x$  by  performing  a clockwise jump of $\delta_k= \| q'_{4k-1} \beta\| - \| q_{4k} \alpha \| >0.$ 
Elements  $n\alpha + m\beta$  in $\phi (G) $ are such that  $ 0 \leq  n < q_{4k}$, or  $0 \leq  m < q'_{4k-1}.$ 
This is the complement of  the set of $(n,m)$
corresponding to Case (1).
Let $F_G$  be defined on $E_{q_{4k+1},q'_{4k+1}} (\alpha, \beta)$  as the  first entering  time  of  $\phi$ to $G$, that is, 
$$F_G (n\alpha + m\beta) : = \min\{ \ell  \ge 0 \,  :  \, \phi^\ell (n\alpha + m\beta) \in G\}.$$
Also, define $S_G$ as the second entering  time  of $\phi$ to $G$:
\begin{align*}
S_G (n\alpha + m\beta) : &= \min\{ \ell \ge F_G (n\alpha + m\beta) +1 \, :\,  \phi^{\ell} (n\alpha + m\beta) \in G\}  \\
&= F_G \left ( \phi^{F_G (n\alpha + m\beta) +1} (n\alpha + m\beta) \right) +  F_G (n\alpha + m\beta).
\end{align*}
We need  to consider the second  entering time to recover an element  located in the clockwise direction.

Let us now define a map   $\widetilde{\Phi}$  on $E_{q_{4k+1},q'_{4k+1}} (\alpha,\beta)$ as follows (see Figure \ref{fig:phi}):

$$\widetilde{\Phi}(n\alpha + m\beta) := \begin{cases}
\phi^{-1} (n\alpha + m\beta) &\text{ if } n\alpha + m\beta \notin \phi (G), \\
\phi^{S_G (n\alpha + m\beta)} (n\alpha + m\beta) &\text{ if } n\alpha + m\beta \in \phi (G). 
\end{cases}$$



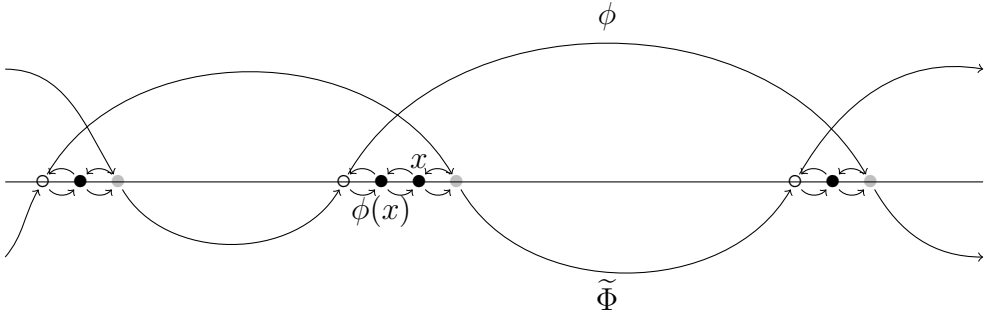
\begin{figure}
\begin{center}
\begin{tikzpicture}[every loop/.style={}]
  \tikzstyle{every node}=[inner sep=0pt]
  
  \node (0) at (0,0) {$\circ$}; 
  \node (1) at (.5,0) {$\bullet$};
  \node (2) at (1,0) {{\color{lightgray}$\bullet$}}; 
  \node (3) at (4,0) {$\circ$}; 
  \node (4) at (4.5,0) {$\bullet$};
 \node (5) at (5,0) {$\bullet$};
   \node [above=5pt] at (5,0) {$x$};
\node (6) at (5.5,0) {{\color{lightgray}$\bullet$}}; 
  \node (7) at (10,0) {$\circ$}; 
  \node (8) at (10.5,0) {$\bullet$};
  \node (9) at (11,0) {{\color{lightgray}$\bullet$}}; 
   \node [below=5pt] at (4.5,0) {$\phi(x)$};
  \draw[->] (-.5,1.5) to [out=0,in=120] (2);
  \draw[->] (0) to [out=60,in=120] (6);
  \draw[->] (3) to [out=60,in=120] (9) node at (7.5,2.2) {$\phi$};
  \draw[->] (7) to [out=60,in=170] (12.5,1.5);
  
  \draw[->] (1) to [out=130,in=50] (0);
  \draw[->] (2) to [out=130,in=50] (1);

  \draw[->] (4) to [out=130,in=50] (3);
  \draw[->] (5) to [out=130,in=50] (4);
  \draw[->] (6) to [out=130,in=50] (5);

  \draw[->] (8) to [out=130,in=50] (7);
  \draw[->] (9) to [out=130,in=50] (8);
  
  \draw  (-.5,0) -- (12.5,0);

  \draw[->] (-.5,-1) to [out=50,in=240] (0);

  \draw[->] (2) to [out=-60,in=240] (3);
  \draw[->] (6) to [out=-60,in=240] (7) node at (7.5,-1.5) {$\widetilde{\Phi}$};
  \draw[->] (9) to [out=-60,in=180] (12.5,-1);
  
  \draw[->] (0) to [out=-50,in=230] (1);
  \draw[->] (1) to [out=-50,in=230] (2);

  \draw[->] (3) to [out=-50,in=230] (4);
  \draw[->] (4) to [out=-50,in=230] (5);
  \draw[->] (5) to [out=-50,in=230] (6);

  \draw[->] (7) to [out=-50,in=230] (8);
  \draw[->] (8) to [out=-50,in=230] (9);

\end{tikzpicture} 
\end{center}
\caption{The points marked by $\circ$ are elements of $G$. Elements of   $\phi (G)$ are   marked in light gray.}
\label{fig:phi}
\end{figure}

The map $\widetilde{\Phi}$ is  a cyclic permutation on $E_{q_{4k+1},q'_{4k+1}} (\alpha,\beta)$. This is illustrated by the skyscraper  tower  construction of  Figure~\ref{fig:tower}
(see for instance  \cite[Page 40]{Petersen}).
 One can  check  that $ \widetilde{\Phi}$  coincides with the function  $\overline{\Phi}$ on $E_{q_{4k+1},q'_{4k+1}} (\alpha,\beta)$.
We will not use this fact in the proof but, as said before,   it aims at providing some insight on the organisation of the cases  that occur in the statement of Proposition \ref{prop:unbounded}.

\begin{figure}
\begin{tikzpicture}[every loop/.style={}]
  \tikzstyle{every node}=[inner sep=-1pt]
  \node (1) at (0,0) {};
  \node (1a) at (1.2,0) {$|$};
  \node (2) at (5,0) {} node [right=4pt] at (5,0) {$\phi (G)$};
  \node (3) at (1.2,1) {};
  \node (3a) at (2.4,1) {$|$};
  \node (4) at (5,1) {};
  \node (9) at (4,2.3) {$\vdots$};
  \node (5) at (3.3,3) {};
  \node (6) at (5,3) {} node [right=4pt] at (5,3) {$\curvearrowright  \phi $};
  \path[-] 
	(1)  edge node [above=4pt] {$G$} (1a)
	(1a)  edge node [above=4pt] {$\uparrow \ \phi$} (2)
	(3)  edge node [above=4pt] {$G$} (3a)
	(3a)  edge node [above=4pt] {$\uparrow \ \phi $} (4)
	(5)  edge node [above=4pt] {$G$} (6);
  \node (10) at (6.2,1.5) {$\searrow$};
  \node (11) at (6.8,3) {};
  \node (11a) at (8,3) {$|$};
  \node (12) at (12.2,3) {} node [right=4pt] at (12.2,3) {$\curvearrowright  \widetilde{ \Phi}$};
  \node (13) at (8,2) {};
  \node (13a) at (9.2,2) {$|$};
  \node (14) at (12.2,2) {};
  \node (19) at (11.2,1.2) {$\vdots$};
  \node (15) at (10.5,0) {};
  \node (15a) at (10.5,0) {};
  \node (16) at (12.2,0) {} node [right=4pt] at (12.2,0) {$G$};
  \path[-] 
	(11)  edge node [below=4pt] {$G$} (11a)
	(11a)  edge node [above=4pt] {$\phi (G)$} (12)
	(13)  edge node [below=4pt] {$G$} (13a)
	(13a)  edge node [above=4pt] {$\uparrow \ \widetilde{ \Phi}$} (14)
	(15)  edge node [above=4pt] {$\uparrow \  \widetilde{\Phi}$} (16);
\end{tikzpicture}
\caption{Each $k$-level of the  tower  is moved  to the level of  index $-k$, with  the indices of  tower on the left  being positive, and negative on the right.
The actions on the rooftops are  $\phi$ and $\widetilde{\Phi}$, respectively.} \label{fig:tower}
\end{figure}

\bigskip

We now come back  to the proof of  Proposition \ref{prop:unbounded}. Let us  count the number of points  $(n,m)$  taking  the same value  $\overline{\Delta}$.
There are 
\begin{enumerate}[label=(\alph*)]
\item   $(q_{4k+1}-q_{4k})q'_{4k}$ points such that $\overline{\Delta}(n,m) = \delta_k$ (Case (1));
\item   $q_{4k-1}   a_{4k+1}q'_{4k-1}$ points such that $\overline{ \Delta}(n,m)= \| q'_{4k+1} \beta \|-a_{4k+1}\delta_k$ (Case (2));
\item  $(q_{4k}- q_{4k-1}) (a_{4k+1}-1)q'_{4k-1} $ points such that $\overline{\Delta}(n,m) = \| q'_{4k+1} \beta \|-(a_{4k+1}-1)\delta_k$ (Case (3));
\item   $2q_{4k-1}q'_{4k-1}$ points  such that $\overline{\Delta}(n,m)= \| q_{4k+1} \alpha \|+ \| q'_{4k+1} \beta \|-2a_{4k+1}\delta_k$ (Case (4) and (6));
\item  $(q_{4k}- 2q_{4k-1})q'_{4k-1}$ points such that $\overline{\Delta}(n,m) = \| q_{4k+1} \alpha \|+\| q'_{4k+1} \beta \|-(2a_{4k+1}-1)\delta_k$ (Case (5));
\item   $2(q_{4k}- q_{4k-1}) q'_{4k-1} $ points  such that $\overline{\Delta}(n,m) = \| q_{4k+1} \alpha \|-a_{4k+1}\delta_k$ (Case (7) and (12));
\item   $2q_{4k} q'_{4k-1} $ points such that $\overline{\Delta}(n,m) = \| q_{4k+1} \alpha \| - (a_{4k+1} +c)\delta_k$ for $1 \le c \le a_{4k+1}-1$ (Case (8) and (11));
\item   another $(q_{4k}- 2q_{4k-1})(q'_{4k+1}- (2a_{4k+1} +1)q'_{4k-1})$ points  such that $\overline{\Delta}(n,m) = \| q_{4k+1} \alpha \|-(2a_{4k+1}-1) \delta_k$ (those points correspond to Case (10), Case (8), with $c= a_{4k+1}-1$, but  we do not take into account Case (11), with $c=1$);
\item  $2q_{4k-1} (q'_{4k+1}- 2a_{4k+1} q'_{4k-1}) $ points  such that $\overline{\Delta}(n,m)= \| q_{4k+1} \alpha \|-2a_{4k+1}\delta_k$ (Case (9)).
\end{enumerate}

As   noticed  in Remark~\ref{rmk:overlap}, there are overlaps between Case (8) and (10) ($c = a_{4k+1}-1$) and between Case (10) and (11) ($c=1$).
There  are $(q_{4k} - 2q_{4k-1}) q'_{4k-1}$ points in both intersections, thus in (h) the total number of points is
\begin{multline*}
(q_{4k}- 2q_{4k-1})(q'_{4k+1}- (2a_{4k+1} -1)q'_{4k-1}) - 2(q_{4k} - 2q_{4k-1}) q'_{4k-1} \\
= (q_{4k}- 2q_{4k-1})(q'_{4k+1}- (2a_{4k+1} +1)q'_{4k-1}) .
\end{multline*}


We denote the sum of all the lengths $\overline{\Delta}$  of  the intervals
 given in the statement  of the proposition  
as 
$$S : = \sum  _{0 \le n \le q_{4k+1}, 0 \le m < q'_{4k+1}} \overline{\Delta}(n,m) = S_0 + S_1 + S_2 + S_3,$$
where
$S_0$ corresponds to  Case (a), $S_1$  to Case (b) and (c),  $S_2$ to Case  (d) and (e), and $S_3$ to the other cases. This yields 
\begin{align*}
S_0 &:= (q_{4k+1}-q_{4k})q'_{4k} \delta_k, \\
S_1 &:= q_{4k-1}   a_{4k+1}q'_{4k-1} (\| q'_{4k+1} \beta \|-a_{4k+1}\delta_k)  \\
&\quad +(q_{4k}- q_{4k-1}) (a_{4k+1}-1)q'_{4k-1} (\| q'_{4k+1} \beta \|-(a_{4k+1}-1)\delta_k),  \\
S_2 &:= 2q_{4k-1}q'_{4k-1} (\| q_{4k+1} \alpha \|+ \| q'_{4k+1} \beta \|-2a_{4k+1}\delta_k)  \\
&\quad +(q_{4k}- 2q_{4k-1})q'_{4k-1} (\| q_{4k+1} \alpha \|+\| q'_{4k+1} \beta \|-(2a_{4k+1}-1)\delta_k).
\end{align*}
Let us prove that $S=1$.
 Since the sum of the  lengths $\overline \Delta$ for Case (g) is
\begin{align*}
&2q_{4k} q'_{4k-1} \sum_{c=1}^{a_{4k+1}-1} \left(  \| q_{4k+1} \alpha \| - (a_{4k+1} +c)\delta_k \right) \\
&= 2q_{4k} q'_{4k-1} (a_{4k+1}-1) \left(  \| q_{4k+1} \alpha \| - a_{4k+1} \delta_k \right) 
- 2q_{4k} q'_{4k-1} \frac{(a_{4k+1}-1)a_{4k+1}\delta_k}{2}  \\
&= 2q_{4k} q'_{4k-1} (a_{4k+1}-1)  \| q_{4k+1} \alpha \| - 3 q_{4k} q'_{4k-1} (a_{4k+1}-1)  a_{4k+1} \delta_k,
\end{align*}
we get
\begin{align*}
S_3 &:=  2(q_{4k}- q_{4k-1}) q'_{4k-1} ( \| q_{4k+1} \alpha \|-a_{4k+1}\delta_k )  \\
&\quad +2q_{4k} q'_{4k-1}  (a_{4k+1}-1) \| q_{4k+1} \alpha \|
- 3q_{4k} q'_{4k-1} a_{4k+1} (a_{4k+1}-1) \delta_k \\
&\quad +(q_{4k}- 2q_{4k-1})(q'_{4k+1}- (2a_{4k+1} +1)q'_{4k-1}) (\| q_{4k+1} \alpha \|-(2a_{4k+1}-1) \delta_k) \\
&\quad +2q_{4k-1} (q'_{4k+1}- 2a_{4k+1} q'_{4k-1}) (\| q_{4k+1} \alpha \|-2a_{4k+1}\delta_k).
\end{align*}
Then we get
\begin{align*}
S_1 &= \left( q_{4k+1} - q_{4k} \right)q'_{4k-1}  \| q'_{4k+1} \beta \|  \\
&\quad -  \left( (a_{4k+1}-1) (q_{4k+1} - q_{4k} )  + a_{4k+1}q_{4k-1} \right) q'_{4k-1} \delta_k, \\
S_2 &= q_{4k} q'_{4k-1} (\| q_{4k+1} \alpha \|+ \| q'_{4k+1} \beta \| ) - (2q_{4k+1} - q_{4k} ) q'_{4k-1} \delta_k, \\ 
S_3 &= q_{4k} ( q'_{4k+1} - q'_{4k-1} ) \| q_{4k+1} \alpha \|   \\
&\quad + \left(  (a_{4k+1} +1)( q_{4k+1} + q_{4k-1} ) -  q_{4k}  \right) q'_{4k-1} \delta_k - ( 2q_{4k+1} - q_{4k} ) q'_{4k+1} \delta_k.
\end{align*}
Therefore, we have 
\begin{align*}
S&= q_{4k}  q'_{4k+1} \| q_{4k+1} \alpha \|+ q_{4k+1} q'_{4k-1}  \| q'_{4k+1} \beta \| \\
&\quad + (q_{4k+1}-q_{4k})q'_{4k} \delta_k +  \left( q_{4k+1} - q_{4k} \right) q'_{4k-1} \delta_k - ( 2q_{4k+1} - q_{4k} ) q'_{4k+1} \delta_k \\ 
&= q_{4k}  q'_{4k+1} \| q_{4k+1} \alpha \|+ q_{4k+1} q'_{4k-1}  \| q'_{4k+1} \beta \| - q_{4k+1} q'_{4k+1} \delta_k \\ 
&= q_{4k}  q'_{4k+1} \| q_{4k+1} \alpha \| + q_{4k+1} q'_{4k-1}  \| q'_{4k+1} \beta \| \\
&\quad - q_{4k+1}  q'_{4k+1} ( \| q'_{4k} \beta \| +\| q'_{4k+1} \beta \| - \| q_{4k}\alpha \|)  \\
&= q'_{4k+1}  \left( q_{4k}  \| q_{4k+1} \alpha \| +q_{4k+1} \| q_{4k}\alpha \| \right)
 - q_{4k+1} \left( q'_{4k}  \| q'_{4k+1} \beta \| + q'_{4k+1} \| q'_{4k} \beta \| \right) \\
&= q'_{4k+1} - q_{4k+1} = 1.
\end{align*}
Hence, the intervals  $\left(n\alpha+m\beta, \overline{\Phi} (n\alpha + m\beta)\right)$ never overlap (as intervals of ${\mathbb T}$),
which implies that $\overline{\Phi} (n\alpha + m\beta)$ is the  neighbor point of $n\alpha + m\beta$, that is,  $  \overline{\Phi}=  \Phi $, which ends the proof of  Proposition \ref{prop:unbounded}.
\end{proof}

\subsection*{End of the proof} 

According to  Cases  (11), (12) of  Proposition~\ref{prop:unbounded}, one has, 
for  $n = c q_{4k}+ 2q_{4k-1}$, $0 \le c \le a_{4k+1}-1$ and $ m = 0$: 
\begin{align*}
&\Phi_{q_{4k+1},q'_{4k+1}}( (c q_{4k}+ 2q_{4k-1})\alpha + 0\beta)  \\
&\quad = ( (2a_{4k+1} -1) q_{4k} - q_{4k+1} + 2q_{4k-1}) \alpha + (2a_{4k+1} - c -1)q'_{4k-1} \beta \\
&\quad = ( q_{4k+1}  - q_{4k}) \alpha + (2a_{4k+1} - c -1)q'_{4k-1} \beta.
\end{align*}

Let $N =  q_{4k+1}=  q'_{4k+1} -1$. 
For each $0 \le c \le a_{4k+1}-1$,   the  following pair of points belongs to $E_{N} (\alpha,\beta)$:
\begin{align*}
 ( c q_{4k}+ 2q_{4k-1}) \alpha &\in E_{N}(\alpha,\beta), \\
 ( q_{4k+1}  - q_{4k})\alpha +  (2a_{4k+1} - c -1)q'_{4k-1} \beta  &\in  E_{N}(\alpha,\beta).
\end{align*}
Since $  E_{N}(\alpha,\beta)   \subset  E_{q_{4k+1},q'_{4k+1}} (\alpha,\beta) $ and the pairs of points  above are adjacent points of $E_{q_{4k+1},q'_{4k+1}} (\alpha,\beta) $,  we have 
 for each $0 \le c \le a_{4k+1}-1$:
\begin{align*}
&( q_{4k+1}  - q_{4k})\alpha +  (2a_{4k+1} - c -1)q'_{4k-1} \beta -  ( c q_{4k}+ 2q_{4k-1}) \alpha  \\
&\quad = ( (2a_{4k+1} -c - 1)q_{4k} - q_{4k+1} )\alpha + (2a_{4k+1} - c -1)q'_{4k-1} \beta \\
&\quad  = \| q_{4k+1} \alpha \| - (2a_{4k+1} -c -1)\delta_k
\in \Delta (E_{N} (\alpha,\beta) ). \end{align*}

Since the sequence  of partial quotients $(a_{4k+1})_k$ goes to infinity, 
we conclude that 
$$ \limsup_{N \to +\infty} \# \Delta E_{N} (\alpha,\beta) = \infty,$$
which  completes the proof of Theorem~\ref{thm} (ii). 
\bibliographystyle{amsalpha}
\bibliography{3gap}

\newcommand{\etalchar}[1]{$^{#1}$}
\providecommand{\bysame}{\leavevmode\hbox to3em{\hrulefill}\thinspace}
\providecommand{\MR}{\relax\ifhmode\unskip\space\fi MR }
\providecommand{\MRhref}[2]{%
  \href{http://www.ams.org/mathscinet-getitem?mr=#1}{#2}
}
\providecommand{\href}[2]{#2}
\begin{thebibliography}{HKWS16}

\bibitem[AB98]{Albe}
Pascal Alessandri and Val\'erie Berth\'e, \emph{Three distance theorems and
  combinatorics on words}, Enseign. Math. (2) \textbf{44} (1998), no.~1-2,
  103--132. \MR{1643286}

\bibitem[ADG{\etalchar{+}}16]{OP:2016}
Faustin Adiceam, David Damanik, Franz G\"ahler, Uwe Grimm, Alan Haynes, Antoine
  Julien, Andr\'es Navas, Lorenzo Sadun, and Barak Weiss, \emph{Open problems
  and conjectures related to the theory of mathematical quasicrystals}, Arnold
  Math. J. \textbf{2} (2016), no.~4, 579--592. \MR{3564887}

\bibitem[BHJ{\etalchar{+}}12]{Bleher:12}
Pavel~M. Bleher, Youkow Homma, Lyndon~L. Ji, Roland K.~W. Roeder, and
  Jeffrey~D. Shen, \emph{Nearest neighbor distances on a circle:
  multidimensional case}, J. Stat. Phys. \textbf{146} (2012), no.~2, 446--465.
  \MR{2873022}

\bibitem[Ble91]{Bleher:1991}
P.~M. Bleher, \emph{The energy level spacing for two harmonic oscillators with
  generic ratio of frequencies}, J. Statist. Phys. \textbf{63} (1991), no.~1-2,
  261--283. \MR{1115584}

\bibitem[BS08]{BringerSchmidt:2008}
Ian Biringer and Benjamin Schmidt, \emph{The three gap theorem and {R}iemannian
  geometry}, Geom. Dedicata \textbf{136} (2008), 175--190. \MR{2443351}

\bibitem[BT02]{BertheTijdeman:2003}
Val\'erie Berth\'e and Robert Tijdeman, \emph{Balance properties of
  multi-dimensional words}, Theoret. Comput. Sci. \textbf{273} (2002), no.~1-2,
  197--224, WORDS (Rouen, 1999). \MR{1872450}

\bibitem[Bug04]{Bugeaud}
Yann Bugeaud, \emph{Approximation by algebraic numbers}, Cambridge Tracts in
  Mathematics, vol. 160, Cambridge University Press, Cambridge, 2004.
  \MR{2136100}

\bibitem[BV00]{BertheVuillon}
Val\'erie Berth\'e and Laurent Vuillon, \emph{Tilings and rotations on the
  torus: a two-dimensional generalization of {S}turmian sequences}, Discrete
  Math. \textbf{223} (2000), no.~1-3, 27--53. \MR{1782038}

\bibitem[CG76]{ChungGraham}
F.~R.~K. Chung and R.~L. Graham, \emph{On the set of distances determined by
  the union of arithmetic progressions}, Ars Combinatoria \textbf{1} (1976),
  no.~1, 57--76. \MR{0412118}

\bibitem[CGVZ02]{CGVZ:2002}
C.~Cobeli, G.~Groza, M.~V\^aj\^aitu, and A.~Zaharescu, \emph{Generalization of
  a theorem of {S}teinhaus}, Colloq. Math. \textbf{92} (2002), no.~2, 257--266.
  \MR{1899442}

\bibitem[Che00]{Chevallier}
Nicolas Chevallier, \emph{Three distance theorem and grid graph}, Discrete
  Math. \textbf{223} (2000), no.~1-3, 355--362. \MR{1782060}

\bibitem[Che07]{Chev:2007}
\bysame, \emph{Cyclic groups and the three distance theorem}, Canad. J. Math.
  \textbf{59} (2007), no.~3, 503--552.

\bibitem[Che14]{Chevallier:2014}
\bysame, \emph{Stepped hyperplane and extension of the three distance theorem},
  Ergodic theory and dynamical systems, De Gruyter Proc. Math., De Gruyter,
  Berlin, 2014, pp.~81--92. \MR{3220099}

\bibitem[Dia17]{Taha}
(Dia)~Taha Diaaeldin, \emph{The three gaps theorems, interval exchange
  transformations, and zippered rectangles}, arXiv:1708.04380, preprint, 2017.

\bibitem[FH95]{FH:95}
Aviezri~S. Fraenkel and Ron Holzman, \emph{Gap problems for integer part and
  fractional part sequences}, J. Number Theory \textbf{50} (1995), no.~1,
  66--86. \MR{1310736}

\bibitem[FS92]{FriedSos:1992}
E.~Fried and Vera~T. S\'os, \emph{A generalization of the three-distance
  theorem for groups}, Algebra Universalis \textbf{29} (1992), no.~1, 136--149.
  \MR{1145560}

\bibitem[GS93]{GeelenSimpson}
J.~F. Geelen and R.~J. Simpson, \emph{A two-dimensional {S}teinhaus theorem},
  Australas. J. Combin. \textbf{8} (1993), 169--197.

\bibitem[Hal65]{Halton:1965}
John~H. Halton, \emph{The distribution of the sequence {$\{n\xi
  \}\,(n=0,\,1,\,2,\,\cdots )$}}, Proc. Cambridge Philos. Soc. \textbf{61}
  (1965), 665--670. \MR{0202668}

\bibitem[HJKW17]{HJKW17}
Alan Haynes, Antoine Julien, Henna Koivusalo, and James Walton,
  \emph{Statistics of patterns in typical cut and project sets},
  arXiv:1702.04041, preprint, 2017.

\bibitem[HKWS16]{HKWS:2016}
Alan Haynes, Henna Koivusalo, James Walton, and Lorenzo Sadun, \emph{Gaps
  problems and frequencies of patches in cut and project sets}, Math. Proc.
  Cambridge Philos. Soc. \textbf{161} (2016), no.~1, 65--85. \MR{3505670}

\bibitem[HM17]{HaynesMarklof17}
A.~Haynes and J.~Marklof, \emph{Higher dimensional {S}teinhaus and {S}later
  problems via homogeneous dynamics}, arXiv:1707.04094, preprint, 2017.

\bibitem[Lan91]{Langevin:1991}
M.~Langevin, \emph{Stimulateur cardiaque et suites de {F}arey}, Period. Math.
  Hungar. \textbf{23} (1991), no.~1, 75--86. \MR{1141354}

\bibitem[Lan95]{Lang}
Serge Lang, \emph{Introduction to {D}iophantine approximations}, second ed.,
  Springer-Verlag, New York, 1995. \MR{1348400}

\bibitem[Lia79]{Liang:1979}
Frank~M. Liang, \emph{A short proof of the {$3d$}\ distance theorem}, Discrete
  Math. \textbf{28} (1979), no.~3, 325--326. \MR{548632}

\bibitem[MS17]{Marklof17}
J.~Marklof and A.~Str\"ombergsson, \emph{The three gap theorem and the space of
  lattices}, arXiv:1612.04906v2 , to appear in the American Mathematical
  Monthly preprint, 2017.

\bibitem[Pet89]{Petersen}
Karl Petersen, \emph{Ergodic theory}, Cambridge Studies in Advanced
  Mathematics, vol.~2, Cambridge University Press, Cambridge, 1989, Corrected
  reprint of the 1983 original. \MR{1073173}

\bibitem[PSZ16]{PSZ:2016}
Gerem\'\i~as Polanco, Daniel Schultz, and Alexandru Zaharescu, \emph{Continuous
  distributions arising from the three gap theorem}, Int. J. Number Theory
  \textbf{12} (2016), no.~7, 1743--1764. \MR{3544409}

\bibitem[S\'58]{Sos:1958}
Vera~T. S\'os, \emph{{On the distribution mod 1 of the sequence $n \alpha$.}},
  {Ann. Univ. Sci. Budap. Rolando E\"otv\"os, Sect. Math.} \textbf{1} (1958),
  127--134.

\bibitem[\'S59]{Swier:59}
S.~\'Swierczkowski, \emph{On successive settings of an arc on the circumference
  of a circle}, Fund. Math. \textbf{46} (1959), 187--189. \MR{0104651}

\bibitem[Sla64]{Slater:1964}
N.~B. Slater, \emph{Distribution problems and physical applications},
  Compositio Math. \textbf{16} (1964), 176--183 (1964). \MR{0174917}

\bibitem[Sla67]{Slater:1967}
Noel~B. Slater, \emph{Gaps and steps for the sequence {$n\theta\ {\rm mod}\
  1$}}, Proc. Cambridge Philos. Soc. \textbf{63} (1967), 1115--1123.
  \MR{0217019}

\bibitem[Su{\'r}58]{Suranyi:1958}
J.~Su{\'r}anyi, \emph{{\"Uber die Anordnung der Vielfachen einer reellen Zahl
  mod 1.}}, {Ann. Univ. Sci. Budap. Rolando E\"otv\"os, Sect. Math.} \textbf{1}
  (1958), 107--111 (German).

\bibitem[Vij08]{Vijay:2008}
Sujith Vijay, \emph{Eleven {E}uclidean distances are enough}, J. Number Theory
  \textbf{128} (2008), no.~6, 1655--1661. \MR{2419185}

\bibitem[vR88]{Ravenstein:1988}
Tony van Ravenstein, \emph{The three gap theorem ({S}teinhaus conjecture)}, J.
  Austral. Math. Soc. Ser. A \textbf{45} (1988), no.~3, 360--370. \MR{957201}

\end{thebibliography}
\end{document}